\documentclass[]{article}

\usepackage{amsmath,amssymb,amsfonts,amsthm,bm}
\usepackage{verbatim}
\usepackage{setspace}
\usepackage{graphicx}
\usepackage{subcaption}
\usepackage{enumitem}
\RequirePackage[colorlinks,citecolor=blue,urlcolor=blue,linkcolor=blue,linktocpage=true]{hyperref}
\usepackage[dvipsnames]{xcolor}
\usepackage{bbm}
\usepackage{algpseudocode}
\usepackage{algorithm}

\usepackage[top=1in, bottom=1in, left=1in, right=1in]{geometry}

\DeclareSymbolFont{rsfs}{U}{rsfs}{m}{n}
\DeclareSymbolFontAlphabet{\mathscrsfs}{rsfs}

\newtheorem{theorem}{Theorem}
\newtheorem{definition}[theorem]{Definition}

\newtheorem{lemma}[theorem]{Lemma}

\newtheorem{remark}{Remark}
\newtheorem{assumption}{Assumption}
\newtheorem{proposition}[theorem]{Proposition}

\renewcommand{\P}{\operatorname{\mathbb{P}}}

\newcommand{\E}{\operatorname{\mathbb{E}}}
\newcommand{\N}{\mathbb{N}}

\newcommand{\R}{\mathbb{R}}

\newcommand{\rmd}{\mathrm{d}}
\newcommand{\sign}{\operatorname{sign}}
\DeclareMathOperator{\var}{var}

\newcommand{\Tr}{\mathrm{Tr}}

\newcommand{\cE}{\mathcal E}
\newcommand{\cB}{\mathcal B}
\newcommand{\cA}{\mathcal A}

\newcommand{\cS}{\mathcal S}
\newcommand{\cX}{\mathcal X}

\newcommand{\cF}{{\mathcal F}}
\newcommand{\cG}{{\mathcal G}}
\def\hF{\widehat{F}}

\newcommand{\bsigma}{{\boldsymbol\sigma}}
\newcommand{\bsig}{{\boldsymbol\sigma}}

\newcommand{\eps}{\varepsilon}

\renewcommand{\hat}{\widehat}

\def\err{{\sf err}}
\def\cuU{\mathscrsfs{U}}

\def\Par{{\sf P}}
\def\smaxcut{\mbox{\tiny \rm max-cut}}
\def\sminbisec{\mbox{\tiny \rm min-bis}}
\def\OPT{{\sf OPT}}
\def\MCUT{{\sf MaxCut}}
\def\MBIS{{\sf MinBis}}

\def\sSK{{\tiny \sf SK}}
\def\GOE{{\sf GOE}}

\def\vE{\vec{E}}
\def\di{{\partial i}}
\def\dj{{\partial j}}
\def\dv{{\partial v}}
\def\Ball{{\sf B}}
\def\sfP{{\sf P}}
\def\root{\o}
\def\tu{\tilde{u}}
\def\bx{{\boldsymbol{x}}}
\def\by{{\boldsymbol{y}}}
\def\bz{{\boldsymbol{z}}}
\def\bu{{\boldsymbol{u}}}
\def\b0{{\boldsymbol{0}}}
\def\hz{\hat{z}}
\def\normal{{\sf N}}
\def\ed{\stackrel{{\rm d}}{=}}

\def\cuF{\mathscrsfs{F}}

\def\bA{{\boldsymbol A}}
\def\bQ{{\boldsymbol Q}}
\def\bfone{{\boldsymbol 1}}
\def\id{{\boldsymbol I}}

\def\bomega{{\boldsymbol \omega}}
\def\bbeta{{\boldsymbol \beta}}
\def\hbz{\hat{\boldsymbol z}}
\def\bT{{\boldsymbol T}}
\def\sT{{\sf T}}
\def\tF{\tilde{F}}
\def\txi{\tilde{\xi}}
\def\<{\langle}
\def\>{\rangle}

\def\cL{{\mathcal L}}

\def\ov{{\overline v}}

\def\tY{\tilde{Y}}
\def\tz{\tilde{z}}
\def\tell{\tilde{\ell}}
\def\root{{o}}
\DeclareMathOperator*{\plim}{p-lim}

\numberwithin{equation}{section}
\numberwithin{theorem}{section}
\begin{document}

 \title{Local algorithms for Maximum Cut and Minimum Bisection on locally treelike regular graphs of large degree}

\author{Ahmed El Alaoui\thanks{Department of Statistics and Data Science, Cornell University. Email: ae333@cornell.edu.}, \;\; Andrea Montanari\thanks{Department of Electrical Engineering and Department of Statistics, Stanford University. Email: montanar@stanford.edu.}, \;\; Mark Sellke\thanks{Department of Mathematics, Stanford University. Email: msellke@gmail.com.}}

\date{}
\maketitle

\begin{abstract}
Given a graph $G$ of degree $k$ over $n$ vertices, we consider the problem of computing a 
near maximum cut or a near minimum bisection in polynomial time. For graphs of girth $2L$,
we develop a local message passing algorithm whose complexity is $O(nkL)$, and that 
achieves near optimal cut values among all $L$-local algorithms. Focusing on max-cut,
the algorithm constructs a cut of value $nk/4+ n\Par_\star\sqrt{k/4}+\err(n,k,L)$,
where $\Par_\star\approx 0.763166$ is the value of the Parisi formula from spin glass theory,
and $\err(n,k,L)=o_n(n)+no_k(\sqrt{k})+n \sqrt{k} o_L(1)$ (subscripts indicate the asymptotic variables).
Our result generalizes to locally treelike graphs, i.e., graphs whose girth becomes $2L$
after removing a small fraction of vertices. 

Earlier work established that, for random $k$-regular graphs, the 
typical max-cut value is $nk/4+ n\Par_\star\sqrt{k/4}+o_n(n)+no_k(\sqrt{k})$. Therefore our
algorithm is nearly optimal on such graphs.
An immediate corollary of this result is that random regular graphs have nearly minimum 
max-cut, and nearly maximum min-bisection among all regular locally treelike graphs.
This can be viewed as a combinatorial version of the near-Ramanujan property of random regular graphs.
\end{abstract}

\section{Introduction}
\label{sec:intro}
Given a $k$-regular graph $G = (V=[n],E)$ over $n$ vertices, the minimum bisection problem aims at finding a partition of the vertex set 
$V= V_+\cup V_-$ into two sets of equal size  as 
to minimize the number of edges across the partition. (To avoid notational burden, we will assume $n$ even throughout.) The partition
can be encoded into a vector $\bsigma = (\sigma_1,\dots,\sigma_n)\in\{-1,+1\}^n$,
which leads to the equivalent optimization problem
\begin{align}
\MBIS(G) = \max\Big\{ \frac{1}{n}\sum_{(i,j)\in  E}\sigma_i\sigma_j:\;\; \sigma_i\in\{-1,+1\}\,,\;\sum_{i=1}^n\sigma_i = 0\Big\}
\, .\label{eq:minbisecOPTdef}
\end{align}
The number of edges across the partition is then given by $(|E|-n \MBIS(G))/2=(nk/4)-n\MBIS(G)/2$.
The maximum cut is defined analogously, except that we want to maximize the
number of edges across the cut, and it is customary not to constrain the cardinality of $V_+$, $V_-$.  We therefore define
\begin{align}
\MCUT(G) = \max\Big\{ -\frac{1}{n}\sum_{(i,j)\in  E}\sigma_i\sigma_j:\;\;\sigma_i\in\{-1,+1\}\Big\}
\, .\label{eq:maxcutOPTdef}
\end{align}
As before, the number of edges cut by such a partition is $(nk/4)+n\MCUT(G)/2$.

For random $k$-regular graphs $G_{n,k}$, it is known based on a combinatorial version of the interpolation method that the maximum cut value $\MCUT(G_{n,k})$ converges in probability as $n \to \infty$ to a limiting non-random degree-dependent constant $c_{\smaxcut}(k)$~\cite{bayati2010combinatorial,gamarnik2014right,salez2016interpolation}. 
The analogous result for the minimum bisection value $\MBIS(G_{n,k})$ is expected but is not known to hold for any fixed value of $k$. See for instance~\cite{salez2016interpolation,huang2018convergence} for variants of this quantity which are known to converge. 
Calling $c_{\sminbisec}(k)$ the putative limiting value of $\MBIS(G_{n,k})$, it was conjectured in \cite{zdeborova2010conjecture} based on the cavity method that 
in fact $c_{\sminbisec}(k)=c_{\smaxcut}(k)$ for all $k\geq 3$. 
The first-order asymptotics for large $k$ of this conjecture were confirmed in \cite{dembo2017extremal}, where it was shown that 
\begin{align}
\MBIS(G_{n,k}) &= \Par_{\star} \sqrt{k} +  o_k(\sqrt{k}) + o_n(1) \, , \\
\MCUT(G_{n,k}) &=  \Par_{\star} \sqrt{k} +  o_k(\sqrt{k}) + o_n(1) \, .
\end{align}
Here the subscripts to the $``o"$ notation indicate the asymptotic variables, and 
\begin{equation}\label{eq:parisiformula}
    \Par_\star=\inf_{\gamma\in\cuU} \Par(\gamma) \, ,
\end{equation}
where the Parisi functional $\Par$ and the function space $\cuU$ are defined as follows.

\begin{definition}

Let $\cuU$ be the space of non-decreasing, right-continuous non-negative functions $\gamma:[0,1)\to \R_+$. The Parisi functional at zero temperature is the function $\Par:\cuU\to\R$ given by
\[
    \Par(\gamma)=\Phi_{\gamma}(0,0)-\frac{1}{2}\int_0^1 t\gamma(t) \rmd t \, ,
\]
where $\Phi_{\gamma}:[0,1)\times \R\to \R$ is the unique solution of the Parisi PDE
\begin{align}
    \label{eq:parisiPDE}\partial_t \Phi_{\gamma}(t,x)+\frac{1}{2}\Big( \gamma(t) (\partial_x \Phi_{\gamma}(t,x))^2 + \partial_{xx}\Phi_{\gamma}(t,x)\Big)&=0 \, ,\\
    \Phi_{\gamma}(1,x)&=|x| \, ,\quad x\in \R \, .
\end{align}
\end{definition}

The constant $\Par_\star$ naturally arises as the limiting ground state energy of 
the Sherrington-Kirkpatrick (SK) model from spin glass theory. In this model, one samples a GOE 
matrix $\bm{A}\in\R^{n\times n}$ ($\bA=\bA^{\sT}$,
 $(A_{ij})_{i<j}\sim_{iid}\normal(0,1/n)$ and  $(A_{ii})_i \sim_{iid}\normal(0,2/n)$, although we remark that the diagonal of $\bA$ does not matter) and then considers the optimization problem
\begin{equation}
    \OPT_{\sSK}(\bm{A}):=\frac{1}{2 n}\max\Big\{\big\langle \bsigma,\bm{A}\bsigma\big\rangle~:~\bsigma\in \{-1,+1\}^n\Big\}.
\end{equation}
The aforementioned result of \cite{dembo2017extremal} was established 
by bounding the difference  $|\E\MCUT(G)/\sqrt{k}- \E \OPT_{\sSK}(\bm{A})|$ (and similarly for $\MBIS(G)$)
via the Lindeberg exchange method.
Furthermore, the celebrated Parisi formula implies that in this dense model, 
the normalized ground state energy converges to $\Par_\star$: 

\begin{theorem}[\cite{talagrand2006parisi,auffinger2017parisi}]
For $\bm{A}\sim \GOE(n)$,
\begin{equation}
    \plim_{n\to\infty}\,  \OPT_{\sSK}(\bm{A})=\Par_\star \, .
\end{equation}
\end{theorem}

The purpose of the present paper is to study efficient algorithms to \emph{compute} a point 
$\bsigma\in \{-1,+1\}^n$ approximately achieving the minimum bisection or maximum cut 
value for regular graphs. 
Our approach builds upon recent algorithmic advances for the SK model. 
Namely, the second author showed in \cite{montanari2021optimization} that an approximate 
optimizer can be constructed in nearly linear time under a certain \emph{no overlap gap} assumption. 
This was generalized in \cite{ams20}, which established the same result under a somewhat more
compact assumption. 
To state this assumption, we first note that~\cite{auffinger2017parisi} shows that the infimum in~\eqref{eq:parisiformula} is achieved by some $\gamma_*\in \cuU$. 
\begin{assumption}\label{ass:frsb}
The function $\gamma_*:[0,1)\to\R_+$ is strictly increasing. 
\end{assumption}
\begin{theorem}[\cite{montanari2021optimization,ams20}]\label{thm:dense}
Under Assumption~\ref{ass:frsb}, for every $\eps >0$, there is an efficient approximate message passing algorithm outputting $\bsig\in \{-1,+1\}^n$ such that 
\[
    \plim_{n\to\infty} \frac{\langle \bsigma, \bm{A} \bsigma\rangle}{2n} \ge \Par_\star -\eps.
\]
\end{theorem}

Our main result is a message passing algorithm on locally treelike regular 
graphs of large degree, also under Assumption~\ref{ass:frsb}. 
(Note that the assumption is about the minimizer of the Parisi formula for the SK model.)
The algorithm is local in the sense defined below. Consequently, our results apply to any
 sequence of $k$-regular graphs $G_{n,k}$ which converges locally-weakly 
 to the $k$-regular tree. A special example is provided of course by $k$-regular random graphs.

Below, we let $(\Omega,\cF,\nu)$ be a probability space and 
$\cG_*(k,\ell;\Omega)$ be the set of rooted $k$-regular graphs of radius (maximum distance from the root)
$\ell$, with vertices carrying labels in $\Omega$. We will consider random variables $(\omega_i)_{i\in V(G)}$ labelling the vertices, with $\omega_i\in\Omega$, $\omega_i\sim\nu$. We denote by $\Ball_i(\ell)$ the ball 
of radius $\ell$ around $i$ in the graph distance on $G$ rooted at $i$, and by $(\Ball_i(\ell),\bomega_{\Ball_i(\ell)})$ the corresponding labeled graph. Here $\bomega_{\Ball_i(\ell)}=(\omega_j)_{j\in \Ball_i(\ell)}$.

\begin{definition}\label{def:localalg}
A randomized algorithm defined on a $k$-regular graph $G$ is $\ell$-\emph{local} if 
there exists a measurable function $F: \cG_*(k,\ell;\Omega)\to \{-1,+1\}$, such that the 
algorithm has the following structure.
First, generate i.i.d.\ random variables $(\omega_i)_{i\in V(G)}$ as above, i.e. with 
$\omega_i\in\Omega$ and $\omega_i\sim\nu$ for some law $\nu$.
Then let $\sigma_i = F(\Ball_i(\ell),\bomega_{\Ball_i(\ell)})$.
Furthermore, we say that such an algorithm is \emph{balanced} if $ \P(\sigma_i=1)=\frac{1}{2}$.
\end{definition}
  
  \begin{definition}\label{def:blockfactor}
A stochastic process $(\xi_v)_{v\in V(T)}\in \cX^{V(T)}$ indexed by the vertices of the infinite $k$-regular tree $T$ and
taking values in a measurable space $\cX$, is a block factor of IID if there exists a  finite $\ell \ge 1$ and a measurable function $F: \cG_*(k,\ell;\Omega)\to \{-1,+1\}$ such that  $\xi_v = F(\Ball_v(\ell),\bomega_{\Ball_v(\ell)})$ for all $v \in V(T)$.
  \end{definition}

While we stated the above definitions in general form for clarity, by the Borel
equivalence theorem, there
is essentially no loss of generality in assuming $(\Omega,\cF,\nu)=(\R,\cB_{\R},\normal(0,1))$,
i.e., taking the $\omega_i$ to be standard Gaussian random variables, which we will do in the remainder of this paper.

\begin{definition}\label{def:locallytreelike}
A finite graph $G$ is $(\ell,\eps)$-locally treelike if $\Ball_i(\ell+1)$ is isomorphic to a tree for at least $(1-\eps)n$ vertices $i\in V(G)$.
\end{definition}

For convenience, we introduce the notation   
\begin{equation}\label{eq:UG}
U_{G} (\bsigma) = \frac{1}{n} \sum_{(i,j)\in E(G)} \sigma_i \sigma_j \, .
\end{equation}

\begin{theorem}
\label{thm:main}
If Assumption~\ref{ass:frsb} holds, then for any $\eta>0$ there exists $\ell_0 \ge 1$ such that for any $\ell \ge \ell_0$ there exists $n_0 = n_0(\ell) \ge 1, k_0 = k_0(\eta,\ell) \ge 2$ such that the following holds for all $k \ge k_0$ and $n \ge n_0$. There exist balanced $\ell$-local algorithms $\cA_{\sminbisec}$ and $\cA_{\smaxcut}$ outputting $\bsig^1$ and $\bsig^2$ respectively such that on any $(\ell,\eta/\sqrt{k})$-locally treelike $k$-regular graph $G$ with $n$ vertices, 
\begin{align}
   \frac{1}{\sqrt{k-1}}U_G(\bsigma^1) \geq \Par_\star - \eta \, ,~~~~\mbox{and}~~~~
   -\frac{1}{\sqrt{k-1}}U_G(\bsigma^2) \geq \Par_\star-\eta\, ,\label{eq:MainThm}
\end{align}
where each event holds with probability at least $1 - Ce^{-n\eta^2/(Ck)}$ where $C = C(\ell,k)>0$.
\end{theorem}
Less formally, in the above result, one should think of $n \to \infty$ first, then $k \to \infty$, then finally $L \to \infty$.
  
Local algorithms for extremal cuts on locally treelike graphs were studied by Lyons
in~\cite{lyons2017factors}. The term \emph{block factor of IID} is introduced in that paper. Lyons
\cite{lyons2017factors} gives a local algorithm achieving  
$U_G(\bsigma^1)/\sqrt{k-1} \ge 2/\pi+o_k(1)$. In fact
his approach can be viewed as a special case of our general strategy: we refer to Section
\ref{sec:special} for further comparison. 
Recent attempts to compute extremal cuts on sparse graphs using quantum algorithms appear in~\cite{barak2021classical,basso2021quantum}. These algorithms are shown to surpass the value $2/\pi$ achieved by the local algorithm of Lyons, but require a quantum computer, or otherwise run in exponential time on a classical one.     
Besides the algorithmic question, there are a number of papers providing upper and lower bounds for the max-cut value and related quantities on sparse random graphs. We refer to the two recent papers~\cite{gamarnik2018max,coja2020ising} and references therein regarding this topic. 

Our algorithm uses a message passing strategy developing the ideas of
\cite{montanari2021optimization,ams20}. However, these papers dealt with dense 
graph structures, and used an approximate message passing (AMP) update 
\cite{bayati2011dynamics} as the underlying iterative scheme. In dealing with sparse
locally treelike graphs, we will consider a message passing algorithm of more classical
type that operates by updating messages associated to the directed edges of $G$
(two directed edges are associated to each edge in $G$). 
The algorithm constructs a martingale on each vertex, coupling martingales on 
different vertices in a non-trivial manner by propagating messages through directed edges. 

\begin{remark}\label{rmk:balanced}
Since our algorithms are local and balanced, an elementary 
concentration bound implies $\sum_{i=1}^n \sigma_i = O(n^{1/2+\eps})$ for any constant $\eps>0$,
with probability converging to one as $n\to\infty$.
Hence, for instance, we can flip $O(n^{0.51})$  variables to construct
an \emph{exact} bisection $\bsig\in \{-1,+1\}^n$ (i.e., such that $\sum_{i=1}^n \sigma_i = 0$).
The resulting cut obeys the guarantees of Theorem~\ref{thm:main}, namely Eq.~\eqref{eq:MainThm}. 
\end{remark}

\paragraph{Organization of the paper.} Section \ref{eq:GeneralMP} describes a general 
message passing scheme, and characterizes the achieved value. In Section \ref{sec:CLT}
we prove that, for large degrees, the messages of this algorithm obey a central limit theorem, 
which simplifies the analysis.  Finally, in Section \ref{sec:ContTime} we take the continuum time 
limit of the previous algorithm and prove Theorem \ref{thm:main}.

\section{A general message passing algorithm}
\label{eq:GeneralMP}

\subsection{Algorithm definition}
\label{sec:AlgoDef-Discrete}

The algorithm depends on two sequences of measurable functions: $F_{\ell}:\R^{\ell+1}\to\R$, 
$H_{\ell}:\R^{\ell+1}\to\R$, for $\ell\ge 0$.

We initialize the algorithm using $(u^0_i)_{i \in V}$ a collection of i.i.d.\ random variables with $u_i^0\sim \normal(0,1)$. The algorithm updates
messages indexed by directed edges $\vE =\{(i\to j): \; (i,j)\in E\}$ (in particular $|\vE|= 2|E|$). We denote by $\partial i$ the set of neighbors of $i$. We will denote the basic messages by 
$(u^{\ell}_{i\to j})_{(i\to j)\in \vE}$, with $\ell\in \{0,1,2,\dots\}$ denoting the iteration number. Messages are initialized by setting $u^0_{i\to j} = u^0_i$,
and updated for $\ell\ge 0$ via
\begin{align}
\begin{split}\label{eq:Iteration}
u^{\ell+1}_{i\to j} &= \frac{1}{\sqrt{k-1}}\sum_{v\in \di \setminus j} A^{\ell-1}_{v\to i} \cdot u^{\ell}_{v\to i}\, ,\\
A^{\ell}_{i\to j} & = F_{\ell}(u^0_{i\to j},\dots,u^{\ell}_{i\to j})\, ,
\end{split}
\end{align}
where, by convention, we set $A^{-1}_{i\to j}=1$.
We further define node quantities 
\begin{align}
u^{\ell+1}_{i} &= \frac{1}{\sqrt{k}}\sum_{v\in \di } A^{\ell-1}_{v\to i} \cdot u^{\ell}_{v\to i}\, . \label{eq:VertexIteration}
\end{align}

Finally, fixing an integer $L \ge 1$ and a `time step' $\delta \in [0,1]$, we construct the outputs of the algorithm via
\begin{align}
z_v^{L} &= \sqrt{\delta} \, \sum_{\ell=1}^LB_{v}^{\ell-1} \cdot u_v^{\ell}\, ,\label{eq:zdef}\\
B_{v}^{\ell} & = H_{\ell}(u^0_{v},\dots,u^{\ell}_v)\, .
\end{align}

Intuitively, we think of $z_v^L$ as a martingale driven by the sequence $(u_v^\ell)_{1 \le \ell \le L}$ when $\delta$ is small and $\delta L$ is constant. The general strategy will be to maximize or minimize the correlation $\E\{z_v^L z_{\ov}^L\}$ over an edge $(v,\ov) \in E$ in the case of minimum bisection or maximum cut respectively, subject to the constraint that $|z_{v}^L| \to 1$ a.s.\ when $\delta \to 0$ and $L \to \infty$ inversely proportionally. Without the last constraint, the process solving this extremal problem is the so-called Gaussian wave function with eigenvalues $\pm 2\sqrt{k-1}$, respectively~\cite{csoka2015invariant}. We consider this case for illustration in the next section.           
Putting back the binary constraint, for $k$ large, we expect this optimization problem to be closely related to its analogue on the complete graph, which was solved in~\cite{montanari2021optimization,ams20} and is given by a stochastic representation of the Parisi formula. Assumption~\ref{ass:frsb} guarantees that this stochastic representation indeed achieves the optimal value $\Par_\star$.       
We will prove a central limit theorem for the sequence $(u_v^\ell)_{1 \le \ell \le L}$ when $k$ is large, and with an appropriate choice of the functions $\{F_\ell, H_\ell\}$, approximate the process $z_v^{\lfloor t/\delta\rfloor}$ by a martingale $M_t = \partial_x \Phi_{\gamma_*}(t,X_t)$ where $\Phi_{\gamma_*}$ is the solution to the Parisi PDE Eq.~\eqref{eq:parisiPDE} with the optimal order parameter $\gamma_*$, Eq.~\eqref{eq:parisiformula}, and $X_t$ is the solution to the associated SDE, see Section~\ref{sec:continuum}. Since $|M_t| \to 1$ a.s.\ as $t\to 1$, this construction approximately solves the extremization problem, and leads to the result of Theorem~\ref{thm:main}. It remains an interesting open problem to solve this extremization problem for fixed $k$, either over processes of the form given in Eq.~\eqref{eq:zdef}, or over (block) factors of IID more generally.        

\subsection{A special case}
\label{sec:special}

Before proceeding with the analysis of the above algorithm, we remark here that with a special choice of the two sequences of functions $\{F_\ell, H_\ell\}$, the above iteration computes a finite approximation of the Gaussian wave function of~\cite{csoka2015invariant} on the infinite tree for any eigenvalue $|\rho| \le 2\sqrt{k-1}$. As a consequence one can recover the result of~\cite[Theorem 4.1]{lyons2017factors} for the bisection and cut values realized by a local algorithm on a $k$-regular random graph. 

For simplicity, assume $G$ is $k$-regular of girth at least $2L+1$ 
(i.e. it is $(L,0)$-locally treelike). Let $A^{\ell}_{i\to j} = 1$ for all $\ell \ge 0$ and $(i,j)\in \vE(G)$. 
We let $B^{\ell}_{i} = \beta_\ell$ for all $i \in V(G)$ and $1 \le \ell \le L$, 
for some deterministic sequence of numbers $(\beta_\ell)_{\ell=1}^{L}$ (not depending on $u_i^0,\cdots,u_i^L$ 
but possibly depending on $L$). We also let $\delta=1$. Then we can verify by induction 
that for every edge $(i,j) \in \vE(G)$ the vector of messages $(u_{i \to j}^{\ell})_{0\le \ell\le L}$ is distributed as 
$\normal(0,\id_{L+1})$. Similarly, for each vertex $i \in V(G)$, $(u_{i}^{\ell})_{0\le \ell \le L} \sim \normal(0,\id_{L+1})$. 
Furthermore, we can write $u_i^\ell$ and $z_i^\ell$ as follows:
 \begin{equation}   \label{eq:z_simple}
 u_i^\ell = \alpha_\ell \sum_{v \in V(G): d(i,v)=\ell} u_v^0 \, , ~~~\mbox{and}~~~
 z_i^\ell = \sum_{s=1}^\ell \beta_s u_i^s \, ,
\end{equation}
with 
 \begin{equation} 
\alpha_0 = 1\, ,~~~~~~\alpha_{\ell} = \frac{1}{\sqrt{k}}\frac{1}{(k-1)^{\frac{\ell-1}{2}}}~~~~~\mbox{for all}~~~~~ 1 \le \ell \le L \,.
 \end{equation} 
Therefore $(z_i^\ell)_{i \in V(G)}$ is a centered Gaussian sequence with $\E\{(z_i^\ell)^2\}=\sum_{s=1}^\ell \beta_s^2$ for all $i \in V(G)$ and $\ell \le L$. A simple computation reveal that the covariance between neighboring vertices $(v,\ov) \in E(G)$ is
 \begin{equation} \label{eq:corr_edge}
 \E\big\{z_v^L z_{\ov}^L\big\}  = \frac{2\sqrt{k-1}}{k} \sum_{\ell=2}^L \beta_{\ell-1} \beta_{\ell}  = \frac{\sqrt{k-1}}{k} \big\langle {\bm \beta} , {\bm T} {\bm \beta} \big\rangle \, , 
\end{equation} 
where ${\bm \beta} = (\beta_1,\cdots,\beta_{L})$ and ${\bm T}$ is a $L \times L$ matrix with entries 
$T_{\ell,\ell+1} = 1$ for $ 1\le \ell \le L-1$, $T_{\ell,\ell-1}=1$ for $2\le \ell \le L$, and zero entries otherwise. 
(This formula will be generalized  to arbitrary choices of $F_\ell$ and $H_\ell$ in Proposition~\ref{propo:EnergyTree} further below.)  

Observe that the matrix ${\bm T}$ has eigenvalues $2\cos(\frac{\ell \pi}{L+1})$ for 
$1\le \ell \le L$. By letting $\bbeta$ be a unit-norm eigenvector of $\bT$ with eigenvalue $\bar{\rho}$, 
we can construct a Gaussian process $\bz_{\bar{\rho}}\in\R^{V(G)}$
with
\begin{align}
\frac{1}{n}\E\big\{\|\bz_{\bar{\rho}}\|^2\big\}=1\, ,~~~~ \frac{1}{n}\E\big\{\big\<\bz_{\bar{\rho}},\bA_G\bz_{\bar{\rho}}\big\>\big\}=\bar{\rho}\sqrt{k-1},~~~~
\forall \, \bar{\rho} \in \cE_L := \Big\{2\cos\big(\frac{\ell \pi}{L+1}\big) \, : \, 1 \le \ell \le L \Big\}\, .
\end{align}
Here $\bA_G$ is the adjacency matrix of $G$.
Since the sets $\cE_L$  form increasingly finer discretizations of the interval 
$[-2,2]$ as $L$ gets large, 
we can approximate Gaussian processes with any correlation $\rho\in [-2\sqrt{k-1},2\sqrt{k-1}]$.
Also, in that limit, the process converges to an eigenfunction of the adjacency 
operator of the infinite tree, known as Gaussian wave function \cite{csoka2015invariant}.

As a special case, we obtain a local algorithm for optimizing the objective function 
$U_G$, Eq.~\eqref{eq:UG}.
Focusing on minimum bisection, let ${\bm \beta}$ be a unit-norm eigenvector of ${\bm T}$ with the largest eigenvalue $2\cos(\frac{\pi}{L+1})$. 
From Eq.~\eqref{eq:corr_edge} we have 
\begin{align*}
\E\big\{z_v^L z_{\ov}^L\big\} = \frac{2\sqrt{k-1}}{k} \cos\Big(\frac{\pi}{L+1}\Big) \, .
\end{align*}
We set $\sigma_i^L = \sign(z_i^L)$ and let $G=G_n$ converge locally-weakly to the infinite regular tree rooted at vertex $o$. Then, with $\rho_L = k \E\{z_v^L z_{\ov}^L\}$, the following holds almost surely for every $L \ge 0$:
\begin{align}
\lim_{n\to\infty}U_{G_{n}} (\bsigma^{L})  &
=\frac{1}{2}\sum_{v \in \partial o } \E\big\{\sigma^L_o \sigma^L_v\big\} \nonumber\\
&= \frac{k}{2} \big(1-2\P\big(z_o^L  z_v^L < 0\big) \big) =  \frac{k}{\pi} \arcsin( \rho_L/k) \, .
\end{align}
As $L \to \infty$, $\rho_L \to  2\sqrt{k-1}$ and we recover the result of~\cite[Theorem 4.1]{lyons2017factors} (appealing to Remark~\ref{rmk:balanced} to construct an exact bisection).
In particular, for large $k$, this algorithm achieves $U_{G_{n}}(\bsigma)/\sqrt{k-1} = 2/\pi+o_{k}(1) + o_{n}(1)$.
The constant $2/\pi \approx 0.636620$ is strictly smaller than $\Par_{\star} \approx 0.763166$.
This is also the value achieved by the Goemans--Williamson semidefinite programming relaxation
\cite{fan2017well}.
The above argument can be adapted in the obvious way to maximum cut, by taking $\bbeta$ corresponding to the lowest eigenvalue of ${\bm T}$.

Before moving to the more general case, let us remark that in the above construction, for any node $v$, the process $(z_v^{\ell})_{\ell \ge 0}$ is a particularly simple martingale: $z_v^{L}$ is as a sum of independent random variables, each one constructed from the randomness available at a particular radius $\ell \le L$ around $v$; see Eq.~\eqref{eq:z_simple}. In order to achieve the higher value $\Par_\star$ we will need to choose the coefficients $A^{\ell}_{i \to j}$ and $B_i^{\ell}$ (or equivalently the functions $F_\ell$ and $H_{\ell}$) in a way keeping memory of past iterates, which drives the martingales towards a binary value $\{-1,+1\}$ and simultaneously maximizes the edge-correlation $\E\{z_vz_{\ov}\}$.        

\subsection{Analysis on the regular tree}
\label{sec:AnalysisTree}

In this section, we analyze the message passing algorithm introduced above on the infinite $k$-regular tree $T=(V(T),E(T))$. We will then 
transfer some of the results to locally treelike graphs.  

We write $\bu^{\le \ell}_{i\to j} = (u^0_{i\to j},\dots,u^{\ell}_{i\to j})$. 
Note that, for any $\ell\ge 0$, the random variables $(\bu^{\le \ell}_{i\to j})_{(i\to j)\in \vE(T)}$ indexed by directed edges in $T$
are identically distributed and their joint distribution is uniquely determined by the sequence of functions $F_{\ell}$.

In particular the distribution of $\bu^{\le \ell}_{i\to j}$ satisfies a recursive distributional equation.
Letting $\bu^{\le \ell}_{\to}$ be a random vector with the same distribution of (any of) the $\bu^{\le \ell}_{i\to j}$, and 
$(\bu^{\le \ell}_{i\to })_{i\ge 1}$ i.i.d.\ copies of the same vector, and $u^0\sim\normal(0,1)$ independent of $(\bu^{\le \ell}_{i\to })_{i\ge 1}$,
we have
\begin{align}\label{eq:RDE}
\bu^{\le \ell+1}_{\to} \ed \cuF_{\ell}\big(u^0,\bu^{\le \ell}_{1\to},\dots,\bu^{\le \ell}_{(k-1)\to}\big)\, ,
\end{align}
where $\cuF_{\ell}(u^0,\bu^{\le \ell}_{1\to},\dots,\bu^{\le \ell}_{(k-1)\to})_0 \equiv u^0$ and, for $0\le t\le\ell-1$,
\begin{align}
\cuF_{\ell}\big(u^0,\bu^{\le \ell}_{1\to},\dots,\bu^{\le \ell}_{(k-1)\to}\big)_{t+1} \equiv \frac{1}{\sqrt{k-1}}\sum_{i=1}^{k-1}F_{t-1}(u^0_{i\to},\dots,u^{t-1}_{i\to})\, u^{t}_{i\to}\, .
\end{align}
We will require the following polynomial growth and  normalization conditions to hold. 
\begin{assumption}[General conditions on updates $F_{\ell}, H_{\ell}$.]\label{ass:Discrete}
We assume the following conditions to hold:
\begin{enumerate}
\item There exists constants $C_{\ell}<\infty$ such that, for all $\ell\ge 0$, $|F_{\ell}(\bu)|\vee|H_{\ell}(\bu)|\le C_{\ell}(1+\|\bu\|^{C_{\ell}})$ for all $\bu\in\R^{\ell+1}$.
\item Let $(u^{t}_{i\to j})_{t\ge 0}$ be the random variables defined by the above recursion, on the $k$-regular tree $T$. Then, for any $\ell\ge 0$,
\begin{align}
\E\{F_{\ell}(u^0_{i\to j},\dots,u_{i\to j}^{\ell})^2\} = 1\, . \label{eq:Normalization}
\end{align}
\end{enumerate}
\end{assumption}
Note that the first assumption implies that, for each $\ell$, $|u_i^{\ell}|$, and $|u_{i\to j}^{\ell}|$ are bounded by polynomials in the Gaussian random variables $\{u_v^0\}_{v\in V(T)}$. In particular, all moments of $|u_i^{\ell}|$, and $|u_{i\to j}^{\ell}|$ are finite.

The second condition is equivalent to $\E\{(A^{\ell}_{i\to j})^2\}=1$. It is immediate to see that this condition is non-empty. 
Note that the functions $F_0,\dots,F_{\ell-1}$ determine the joint distribution of $u^0_{i\to j},\dots,u^{\ell+1}_{i\to j}$. If the normalization condition is not satisfied by these functions, we can extend the sequence, for instance by starting with an arbitrary polynomially bounded function $\hF_{\ell}:\R^{\ell+1}\to\R$.
We could then satisfy Eq.~(\ref{eq:Normalization}) by setting
\begin{align}
\label{eq:normalizeF}F_{\ell}(\tu^0_{i\to j},\dots,\tu_{i\to j}^{\ell}) 
= \frac{\hF_{\ell}(\tu^0_{i\to j},\dots,\tu_{i\to j}^{\ell})}
{\E\{\hF_{\ell}(u^0_{i\to j},\dots,u_{i\to j}^{\ell})^2\}^{1/2}}\, 
\end{align}
so long as the denominator is non-zero. 

The main result of this section is a computation of the correlation of $z^L_v$ and $z^L_{\ov}$ 
over an edge $(v,\ov) \in E(T)$ achieved by the local algorithm of Section \ref{sec:AlgoDef-Discrete} on the regular tree.
\begin{proposition}\label{propo:EnergyTree}
Consider the algorithm of Section \ref{sec:AlgoDef-Discrete}, under Assumption \ref{ass:Discrete}. If $(v,\ov)\in E(T)$, then
\begin{align*}
\E\{z^L_{v}z^L_{\ov}\} &= \frac{2\sqrt{k-1}}{k}\sum_{\ell=2}^L\E\Big\{A^{\ell-2}_{\ov\to v}(u^{\ell-1}_{\ov\to v})^2 B^{\ell-1}_vB^{\ell-2}_{\ov}\Big\} \, \delta \\
&~~~+ \frac{1}{k}\sum_{\ell=1}^L\E \Big\{B_v^{\ell-1}  B_{\ov}^{\ell-1} A_{v\to \ov }^{\ell-2} A_{\ov \to v}^{\ell-2}  u_{v \to \ov}^{\ell-1} u_{\ov \to v}^{\ell-1}  \Big\}\, \, \delta \, .
%
\end{align*}
\end{proposition}


The rest of this section is devoted to the proof of Proposition~\ref{propo:EnergyTree}. 
The arguments, somewhat routine, exploit the conditional independence structure of the algorithm on the tree $T$. 

In order to proceed we require the following notation. We denote by $\Ball_i(\ell)$ the subgraph of $T$ induced by vertices whose distance from $i$ is at most $\ell$, and by
$\Ball_{i\to j}(\ell)$ the subgraph induced by all the vertices that can be reached from  $i$
via a path of length at most $\ell$ that does not  pass through $j$. With an abuse of notation, we will use $\Ball_i(\ell)$, $\Ball_{i\to j}(\ell)$ 
also to denote the corresponding sets of vertices.

It is useful to define certain $\sigma$-algebras:
\begin{itemize}
\item $\cF^{\ell,+}_{i\to j} \equiv \sigma(u^{0}_{i\to j},\dots,u^{\ell}_{i\to j})$, $\cF^{\ell,-}_{i\to j} \equiv \sigma(u^{1}_{i\to j},\dots,u^{\ell}_{i\to j})$.
\item $\cF^{\ell,+}_{i} \equiv \sigma(u^{0}_{i},\dots,u^{\ell}_{i})$, $\cF^{\ell,-}_{i} \equiv \sigma(u^{1}_{i},\dots,u^{\ell}_{i})$.
\item $\cG^{\ell,+}_{i \to j} \equiv \sigma(\{u^0_v:\; v\in \Ball_{i\to j}(\ell)\})$, $\cG^{\ell,-}_{i \to j}\equiv \sigma(\{u^0_v:\; v\in \Ball_{i\to j}(\ell)\setminus\{i\}\})$.
\item $\cG^{\ell,+}_{i} \equiv \sigma(\{u^0_v:\; v\in \Ball_{i}(\ell)\})$, $\cG^{\ell,-}_{i} \equiv \sigma(\{u^0_v:\; v\in \Ball_{i}(\ell)\setminus\{i\}\})$.
\end{itemize}
For $\ell=-1$ we stipulate that all of these $\sigma$-algebras coincide with the trivial 
one.

\begin{lemma}\label{lemma:Sigma}
The following relations hold between the $\sigma$-algebras defined above:
\begin{enumerate}
\item[$(a)$] For all $\ell\ge 1$, $u^\ell_{i\to j}\in \cG^{\ell,-}_{i\to j}$, $u^\ell_{i}\in \cG^{\ell,-}_{i}$. 
\item[$(b)$] $\cF^{\ell,\pm}_{i\to j}\subseteq \cG^{\ell,\pm}_{i\to j}$ and $\cF^{\ell,\pm}_{i}\subseteq \cG^{\ell,\pm}_{i}$.
\item[$(c)$] If the shortest path between $i_1$ and $i_2$ passes through vertices $j_1$, $j_2$, then $\cG^{\ell_1,\pm}_{i_1\to j_1}$ and $\cG^{\ell_2,\pm}_{i_2\to j_2}$ are independent.
As a consequence, $\cF^{\ell_1,\pm}_{i_1\to j_1}$ and $\cF^{\ell_2,\pm}_{i_2\to j_2}$ are independent.
(This also holds for $(i_1,i_2)\in E(T)$, and $i_2=j_1$, $i_1=j_2$.)
\item[$(d)$] $\cF^{\ell+1,-}_{i\to j}\subseteq \sigma(\cup_{v\in \di\setminus j}\cF^{\ell,+}_{v\to i})$.
\item[$(e)$] For $j\in \di$, $\cG^{\ell,+}_i = \sigma(\cG^{\ell,+}_{i\to j}\cup \cG^{\ell-1,+}_{j\to i})$ and 
$\cG^{\ell,+}_{i\to j}= \sigma(\sigma(\{u_i^0\})\cup_{v\in\di\setminus j}\cG^{\ell-1,+}_{v\to i})$.
\end{enumerate}
\end{lemma}
\begin{proof}
\noindent${\boldsymbol (a)}$ Consider the statement $u^\ell_{i\to j}\in \cG^{\ell,-}_{i\to j}$, since the other one is proved analogously. 
We can proceed by induction over $\ell$. For $\ell=1$, we have  $u^\ell_{i\to j}= (k-1)^{-1/2}\sum_{v\in\di\setminus j} u^0_v$ and
$\cG^{\ell,-}_{i\to j}= \sigma(\{u^0_v:v\in\di\setminus j\})$. Assume that the statement holds for some $\ell\ge 1$,
and consider the definition (\ref{eq:Iteration}) of $u^{\ell+1}_{i\to j}$. By the induction hypothesis, we have $A^{\ell-1}_{v\to i}u^{\ell}_{v\to i}\in \cG_{v\to i}^{\ell}$.
Since by construction $\cG_{v\to i}^{\ell}\subseteq \cG_{i\to j}^{\ell+1}$, the claim follows.

\noindent${\boldsymbol (b)}$ Consider the inclusion $\cF^{\ell,+}_{i\to j}\subseteq \cG^{\ell,+}_{i\to j}$, since the others are 
proved analogously. The proof is by induction over $\ell$. For $\ell=0$ we have $\cF^{0,+}_{i\to j}= \cG^{0,+}_{i\to j}= \sigma(u^0_i)$.
Assume that the claim holds for a certain $\ell\ge 0$. Of course $\cF^{\ell+1,+}_{i\to j}=\sigma(\sigma(\{u^{\ell+1}_{i\to j}\})\cup \cF^{\ell,+}_{i\to j})$.
Since by the induction hypothesis $\cF^{\ell,+}_{i\to j}\subseteq\cG^{\ell,+}_{i\to j}\subseteq\cG^{\ell+1,+}_{i\to j}$ and by point $(a)$ we have
$u^{\ell+1}_{i\to j}\in \cG^{\ell+1,-}_{i\to j}\subseteq \cG^{\ell+1,+}_{i\to j}$, the claim follows.

\noindent${\boldsymbol (c)}$ This is immediate, since $\Ball_{i_1\to j_1}(\ell_1)$ and $\Ball_{i_2\to j_2}(\ell_2)$ are disjoint.

\noindent${\boldsymbol (d)}$ It is sufficient to prove that, for each $\ell$, $u^{\ell+1}_{i\to j}\in \sigma(\cup_{v\in \di\setminus j}\cF^{\ell,+}_{v\to i})$.
This follows as above from the definition (\ref{eq:Iteration}).

\noindent${\boldsymbol (e)}$ This follows from $\Ball_i(\ell) = \Ball_{i\to j}(\ell)\cup \Ball_{j\to i}(\ell-1)$,
and $\Ball_{i\to j}(\ell) =  \{i\}\cup_{v\in\di\setminus j}\Ball_{v\to i}(\ell-1)$.
\end{proof}

The next lemma is a standard consequence of the monotone class theorem (see e.g., \cite[Theorem 5.2.2]{durrett2019probability}).
\begin{lemma}\label{lemma:RemarkSigma}
Let $\cF_1,\dots, \cF_m,\cA$ be $\sigma$-algebras, and $\cF=\sigma(\cF_1\cup\cdots\cup\cF_m)\subseteq\cA$. Assume $X\in \cA$,
and that, for any $Y_1\in \cF_1$, \dots, $Y_m\in \cF_m$, $\E\{XY_1\cdots Y_m\}=0$ (provided the expectation exists). 
Then $\E\{XY\}=0$ for any $Y\in \cF$ such that the expectation exists.
(The same statement remains true if $Y_i$'s are bounded non-negative.) 
\end{lemma}

\begin{lemma}\label{lemma:Umoment}
Let $(i,j)\in E(T)$, $\ell\ge -1$, and $Y_{i}^{\ell}\in \cG_{i}^{\ell,+}$. Then
\begin{align}
\E\big\{u^{\ell+1}_{i\to j}Y_i^{\ell}\big\} = 0\, ,\;\;\;  \E\big\{u^{\ell+1}_{i}Y_i^{\ell}\big\} = 0\, .\label{eq:U1moment}
\end{align}
\end{lemma}
\begin{proof}
We will limit ourselves to proving the equations for $u^{\ell}_{i\to j}$, since the ones for $u^{\ell}_{i}$ are completely analogous.

By Lemma~\ref{lemma:Sigma}.$(e)$ and Lemma~\ref{lemma:RemarkSigma}, it is sufficient to prove the claim for
$Y_i^{\ell} = Y^{\ell}_{i\to j}  Y^{\ell-1}_{j\to i}$, where $Y^{\ell}_{i\to j}\in \cG_{i\to j}^{\ell}$, and $Y^{\ell-1}_{j\to i}\in \cG_{j\to i}^{\ell-1}$.
Since by Lemma  \ref{lemma:Sigma}.$(c)$, we have $\E\big\{u^{\ell+1}_{i\to j}Y_i^{\ell}\big\} =\E\big\{u^{\ell+1}_{i\to j}Y_{i\to j}^{\ell}\big\} \E\big\{Y_{j\to i}^{\ell-1}\big\}$,
it is sufficient to prove that $\E\big\{u^{\ell+1}_{i\to j}Y_{i\to j}^{\ell}\big\}=0$. We do this by induction over $\ell$. For $\ell =-1$, the claim is trivial, because
$\cG_{i\to j}^{-1,+}$ is the trivial $\sigma$-algebra, and therefore $Y_{i\to j}^{-1}$ is a constant. Therefore
$\E\big\{u^{0}_{i\to j}Y_{i\to j}^{-1}\big\} = \E\{u^0_i\} Y_{i\to j}^{-1}=0$.

For the induction step, consider any $\ell\ge 0$. Note that by Lemma \ref{lemma:RemarkSigma}, we can assume without loss of generality $Y^{\ell}_{i\to j}=Y^0_i
\prod_{v\in\di\setminus j}Y^{\ell-1}_{v\to i}$, where $Y^{0}_i\in \cG^{0,+}_i$ (namely, $Y^{0}_i$ is a function of $u^0_i$), and $Y^{\ell-1}_{v\to i}\in\cG^{\ell-1,+}_{v\to i}$.
Using the update (\ref{eq:Iteration}), we get
\begin{align*}
\E\big\{u^{\ell+1}_{i\to j}Y_{i\to j}^{\ell}\big\} &= \frac{1}{\sqrt{k-1}}\sum_{m\in\di\setminus j}\E\big\{A^{\ell-1}_{m\to i}u^\ell_{m\to i} Y^0_i\prod_{v\in\di\setminus j} Y^{\ell-1}_{v\to i}\big\}\\
&= \frac{1}{\sqrt{k-1}}\sum_{m\in\di\setminus j}\E\big\{A^{\ell-1}_{m\to i}u^\ell_{m\to i} Y^{\ell-1}_{m\to i} \big\}\, \E\big\{Y^0_i\prod_{v\in\di\setminus \{j,m\}} Y^{\ell-1}_{v\to i}\big\}\\
& = 0\, ,
\end{align*}
where in the last step we used the induction hypothesis and the fact that $A^{\ell-1}_{m\to i}Y^{\ell-1}_{m\to i} \in \cG_{m\to i}^{\ell-1}$.
\end{proof}

\begin{lemma}\label{lemma:Neighbors-1}
Let $(v,\ov)\in E(T)$ and assume $Y_v^{t-1}\in \cG_{v}^{t-1,+}$, $\tY_{\ov}^{s}\in \cG_{\ov}^{s,+}$. If 
$t\ge s\vee 1$, then
\begin{align}
\E\big\{u^t_vY_v^{t-1}\tY^s_{\ov}\big\} = \frac{1}{\sqrt{k}} \E\big\{A^{t-2}_{\ov\to v}u^{t-1}_{\ov\to v}Y^{t-1}_{v}\tY^s_{\ov}\big\}\, .
\end{align}
\end{lemma}
\begin{proof}
Using Eq.~(\ref{eq:Iteration}), we get 
\begin{align}
\E\big\{u^t_vY_v^{t-1}\tY^s_{\ov}\big\} =  \frac{1}{\sqrt{k}}\sum_{j\in\dv\setminus\ov}\E\big\{A^{t-2}_{j\to v}u^{t-1}_{j\to v}Y_v^{t-1}\tY^s_{\ov}\big\}+
\frac{1}{\sqrt{k}} \E\big\{A^{t-2}_{\ov\to v}u^{t-1}_{\ov\to v}Y^{t-1}_{\ov\to v}\tY^s_{\ov}\big\} \, .
\end{align}
The proof follows by showing that the terms in the first sum vanish. Throughout the rest of the proof, we will denote by $Y^\ell_{i}$, $\tY^\ell_{i}$  random variables measurable
with respect to $\cG^{\ell,+}_i$, and  by $Y^\ell_{i\to j}$, $\tY^\ell_{i\to j}$ random variables measurable with respect to $\cG^{\ell,+}_{i\to j}$.
Fix $j\in\dv\setminus\ov$. 
By Lemma \ref{lemma:Sigma}.$(e)$ and Lemma \ref{lemma:RemarkSigma}, we can assume 
$Y_v^{t-1} = Y_{j\to v}^{t-2} Y_{v\to j}^{t-1}$, $\tY^s_{\ov}=\tY^s_{\ov\to v}\tY^{s-1}_{v\to \ov}$, and 
$\tY^{s-1}_{v\to \ov} = \tY^0_v\prod_{i\in\dv\setminus\ov} \tY^{s-2}_{i\to v}$. Therefore
\begin{align*}
\E\big\{A^{t-2}_{j\to v}u^{t-1}_{j\to v}Y_v^{t-1}\tY^s_{\ov}\big\} &= \E\big\{A^{t-2}_{j\to v}u^{t-1}_{j\to v}
Y_{j\to v}^{t-2} Y_{v\to j}^{t-1}\tY^s_{\ov\to v}\tY^{s-1}_{v\to \ov}\big\}\\
&= \E\Big\{A^{t-2}_{j\to v}u^{t-1}_{j\to v}Y_{j\to v}^{t-2} Y_{v\to j}^{t-1}\tY^s_{\ov\to v}\tY^0_v\prod_{i\in\dv\setminus\ov} \tY^{s-2}_{i\to v}\Big\}\\
&\stackrel{(1)}{=} \E\big\{A^{t-2}_{j\to v}u^{t-1}_{j\to v}Y_{j\to v}^{t-2} \tY^{s-2}_{j\to v}\big\}\,  \E\Big\{Y_{v\to j}^{t-1}\tY^s_{\ov\to v}\tY^0_v\prod_{i\in\dv\setminus\{\ov,j\}} 
\tY^{s-2}_{i\to v}\Big\}\\
& \stackrel{(2)}{=}  0 \, ,
\end{align*}
where we used Lemma \ref{lemma:Sigma}.$(c)$ in $(1)$, and Lemma \ref{lemma:Umoment} which implies
 $\E\big\{A^{t-2}_{j\to v}u^{t-1}_{j\to v}Y_{j\to v}^{t-2} \tY^{s-2}_{j\to v}\big\} = 0$ for $t\ge s$ (step $(2)$).
\end{proof}

\begin{lemma}\label{lemma:T2}
Let $(v,\ov)\in E(T)$ and assume $Y_v^{t-1}\in \cG_{v}^{t-1,+}$, $\tY_{\ov}^{s}\in \cG_{\ov}^{s,+}$. If 
$t\ge s+2$, $s\ge 0$, then
\begin{align}
\E\big\{u^t_vY^{t-1}_v\tY_{\ov}^{s}\big\} = 0\, .
\end{align}
\end{lemma}
\begin{proof}
By Lemma \ref{lemma:RemarkSigma} and Lemma \ref{lemma:Sigma}.$(e)$, we can assume, without loss of generality 
$Y^{t-1}_{v} = Y^{t-1}_{v\to\ov}Y^{t-2}_{\ov\to v}$ and $\tY^{s}_{\ov} = \tY^{s}_{\ov\to v}\tY^{s-1}_{v\to\ov}$.
By Lemma \ref{lemma:Neighbors-1}, we get
\begin{align*}
\E\big\{u^t_vY_v^{t-1}\tY^s_{\ov}\big\} &= \frac{1}{\sqrt{k}} \E\big\{A^{t-2}_{\ov\to v}u^{t-1}_{\ov\to v}Y^{t-1}_{v}\tY^s_{\ov}\big\}\\
&= \frac{1}{\sqrt{k}} \E\big\{A^{t-2}_{\ov\to v}u^{t-1}_{\ov\to v} Y^{t-1}_{v\to\ov}Y^{t-2}_{\ov\to v}  \tY^{s}_{\ov\to v}\tY^{s-1}_{v\to\ov}\big\}\\
&\stackrel{(1)}{=} \frac{1}{\sqrt{k}} \E\big\{A^{t-2}_{\ov\to v}u^{t-1}_{\ov\to v} Y^{t-2}_{\ov\to v}\tY^{s}_{\ov\to v}\big\} \E\big\{\tY^{s-1}_{v\to\ov}Y^{t-1}_{v\to\ov}\big\}\\
&\stackrel{(2)}{=}   0\, ,
\end{align*}
where in step $(1)$ we use Lemma \ref{lemma:Sigma}.$(c)$ and in step $(2)$ we used Lemma~\ref{lemma:Umoment},
together with the fact that $A^{t-2}_{\ov\to v}Y^{t-2}_{\ov\to v}\tY^{s}_{\ov\to v}\in \cG^{t-2}_{\ov\to v}$.
\end{proof}

\begin{lemma}\label{lemma:T1}
Let $(v,\ov)\in E(T)$ and assume (for some $t\ge 2$) $Y_v^{t-1}\in \cG_{v}^{t-1,+}$, $\tY_{\ov}^{t-2}\in \cG_{\ov}^{t-2,+}$. Then
\begin{align}
\E\big\{u^t_vY^{t-1}_vu_{\ov}^{t-1}\tY_{\ov}^{t-2}\big\} = \frac{\sqrt{k-1}}{k}\, \E\big\{A^{t-2}_{\ov\to v}(u^{t-1}_{\ov\to v})^2 Y^{t-1}_v\tY^{t-2}_{\ov}\big\}\, .
\end{align}
\end{lemma}
\begin{proof}
By Lemma \ref{lemma:Neighbors-1}, and using the fact that $u^{t-1}_{\ov}\tY^{t-2}_{\ov} \in \cG^{t-1,+}_{\ov}$, we have
\begin{align*}
\E\big\{u^t_vY^{t-1}_vu_{\ov}^{t-1}\tY_{\ov}^{t-2}\big\} = \frac{1}{\sqrt{k}}\, \E\big\{A^{t-2}_{\ov\to v}u^{t-1}_{\ov\to v} Y^{t-1}_vu^{t-1}_{\ov}\tY^{t-2}_{\ov}\big\}\, .
\end{align*}
On the other hand Eqs.~\eqref{eq:Iteration} and \eqref{eq:VertexIteration} imply, for any $j\in\di$,
\begin{equation}\label{eq:decomp_m}
u^{t-1}_{\ov} = \sqrt{\frac{k-1}{k}} \, u^{t-1}_{\ov\to v}+\frac{1}{\sqrt{k}} A^{t-3}_{v\to \ov}u^{t-2}_{v\to \ov}\, .
\end{equation}
Hence
\begin{align*}
\E\big\{u^t_vY^{t-1}_vu_{\ov}^{t-1}\tY_{\ov}^{t-2}\big\} =& \frac{\sqrt{k-1}}{k}\, \E\big\{A^{t-2}_{\ov\to v}(u^{t-1}_{\ov\to v})^2 Y^{t-1}_v\tY^{t-2}_{\ov}\big\}\\
&+ \frac{1}{k}\, \E\big\{A^{t-2}_{\ov\to v}u^{t-1}_{\ov\to v}A^{t-3}_{v\to \ov}u^{t-2}_{v\to\ov} Y^{t-1}_v\tY^{t-2}_{\ov}\big\}\, .
\end{align*}
The proof follows by showing that the second term on the right-hand side vanishes. To this end, by Lemmas \ref{lemma:Sigma} and \ref{lemma:RemarkSigma}
we can assume $Y^{t-1}_v = Y^{t-1}_{v\to \ov}Y^{t-2}_{\ov\to v}$ and $\tY^{t-2}_{\ov} = \tY^{t-3}_{v\to \ov}\tY^{t-2}_{\ov\to v}$. Using once more Lemma \ref{lemma:Sigma},
we get
\begin{align}
\E\big\{A^{t-2}_{\ov\to v}u^{t-1}_{\ov\to v}A^{t-3}_{v\to \ov}u^{t-2}_{v\to\ov} Y^{t-1}_v\tY^{t-2}_{\ov}\big\}= 
\E\big\{A^{t-2}_{\ov\to v}u^{t-1}_{\ov\to v} Y^{t-2}_{\ov\to v} \tY^{t-2}_{\ov\to v}  \big\} \E\big\{A^{t-3}_{v\to \ov}u^{t-2}_{v\to\ov} Y^{t-1}_{v\to \ov} \tY^{t-3}_{v\to \ov}\big\}  =0\, .
\end{align}
In the last step we use the fact that 
$A^{t-2}_{\ov\to v}Y^{t-2}_{\ov\to v} \tY^{t-2}_{\ov\to v}\in \cG^{t-2,+}_{\ov\to v}$, together with Lemma  \ref{lemma:Umoment}.
\end{proof}

\begin{lemma}\label{lemma:Diagonal}
Let $(v,\ov)\in E(T)$ and assume $Y_v^{\ell}\in \cG_{v}^{\ell,+}$, $\tY_{\ov}^{\ell}\in \cG_{\ov}^{\ell,+}$ for each $\ell\ge 0$. Then, for each $t\ge 1$,
\begin{align}
\E\big\{u_v^t u_{\ov}^t Y_v^{t-1} \tY_{\ov}^{t-1}\big\} = \E\big\{A_{v\to \ov}^{t-2} u_{v\to \ov}^{t-1}  A_{\ov\to v}^{t-2}  u_{\ov \to v}^{t-1} Y_v^{t-1}\tY_{\ov}^{t-1}\big\}\, .
%
\end{align}
\end{lemma}
\begin{proof}
Since $u_{\ov}^t \tY_{\ov}^{t-2}\in \cG^{t,+}_{\ov}$, we apply Lemma \ref{lemma:Neighbors-1} to obtain
\begin{align*}
\E\big\{u_v^t u_{\ov}^t Y_v^{t-1} \tY_{\ov}^{t-1}\big\} &=
\frac{1}{\sqrt{k}}\,\E\big\{A_{\ov\to v}^{t-2}u_{\ov\to v}^{t-1} u_{\ov}^t Y_v^{t-1} \tY_{\ov}^{t-1}\big\}\nonumber\\
& = \frac{1}{k}\,\E\big\{A_{\ov\to v}^{t-2}u_{\ov\to v}^{t-1} A^{t-2}_{v\to \ov}u_{v\to \ov}^{t-1} Y_v^{t-1} \tY_{\ov}^{t-1}\big\}+
\frac{\sqrt{k-1}}{k}\,\E\big\{A_{\ov\to v}^{t-2}u_{\ov\to v}^{t-1} u_{\ov\to v}^t Y_v^{t-1} \tY_{\ov}^{t-1}\big\} \, ,
%
\end{align*}
where the last line follows by using Eq.~\eqref{eq:decomp_m} with a time index $t$ instead of $t-1$. 
It remains to show that the second term in the above display vanishes.
 By Lemma \ref{lemma:RemarkSigma} and Lemma~\ref{lemma:Sigma}.$(e)$, we can assume, without loss of generality,  $Y_v^{t-1} = Y_{v\to \ov}^{t-1}Y_{\ov\to v}^{t-2}$ and $\tY_{\ov}^{t-1}=\tY_{\ov\to v}^{t-1}\tY_{v\to \ov}^{t-2}$.
Substituting in the second term above and using independence (Lemma~\ref{lemma:Sigma}.$(c)$), we get
\begin{align*}
\E\big\{A_{\ov\to v}^{t-2}u_{\ov\to v}^{t-1} u_{\ov\to v}^t Y_v^{t-1} \tY_{\ov}^{t-1}\big\}
&= \E\big\{A_{\ov\to v}^{t-2}u_{\ov\to v}^{t-1} u_{\ov\to v}^t Y_{\ov\to v}^{t-2} \tY_{\ov \to v}^{t-1}\big\} \cdot 
\E\big\{Y_{v \to \ov}^{t-1} \tY_{v \to \ov}^{t-2}\big\} \, .
\end{align*}
Since $A^{t-2}_{\ov\to v}u^{t-1}_{\ov \to v} Y^{t-2}_{\ov \to v} \tY^{t-1}_{\ov \to v} \in \cG^{t-1,+}_{\ov\to v}$, the first expectation on the right-hand side is zero by Lemma~\ref{lemma:Umoment}.  
%
\end{proof}

We can now collect and apply the previous lemmas to obtain a statement about the correlation
achieved by the local algorithm of Section \ref{sec:AlgoDef-Discrete}.
\begin{proof}[Proof of Proposition~\ref{propo:EnergyTree}]
Using Eq.~\eqref{eq:zdef}, we get 
\begin{align*}
\E\{z^L_vz^L_{\ov}\} &= 2\sum_{\ell=1}^L\sum_{\tell=1}^{\ell-1} \E\{u_v^{\ell}u_{\ov}^{\tell}B_v^{\ell-1}B_{\ov}^{\tell-1}\} \, \delta + \sum_{\ell=1}^L\E\{u_v^{\ell}u_{\ov}^{\ell}B_v^{\ell-1}B_{\ov}^{\ell-1}\} \, \delta\\
& =  2\sum_{\ell=2}^L\E\{u_v^{\ell}u_{\ov}^{\ell-1}B_v^{\ell-1}B_{\ov}^{\ell-2}\}\, \delta + \sum_{\ell=1}^L\E\{u_v^{\ell}u_{\ov}^{\ell}B_v^{\ell-1}B_{\ov}^{\ell-1}\}\, \delta\, ,
\end{align*}
where the last equality follows from Lemma \ref{lemma:T2}, using the fact that $B_{i}^{\ell}\in \cG^{\ell,+}_{i}$.
We next apply Lemma \ref{lemma:T1} to the first sum, and Lemma \ref{lemma:Diagonal} to the second one, thus getting the desired claim.
\end{proof}

\section{A central limit theorem for large degree}
\label{sec:CLT}
We consider the situation where the degree $k$ is large. 
We will show in this case that the vector of messages $(u_{i \to j}^{\ell})_{1\le \ell \le L}$ for each directed edge 
$(i,j) \in \vE(T)$ is approximately Gaussian. In order to avoid inessential technicalities,
we make the following simplifying assumption. It will always hold in our applications thanks to Proposition~\ref{prop:phireg}.
\begin{assumption}[Boundedness of $F_{\ell}, H_{\ell}$.]\label{ass:Bounded}
We assume that there exists a constant $K<\infty$ such that, for all $\ell\ge 0$,
 $|F_{\ell}(\bu)|\vee|H_{\ell}(\bu)|\le K$ for all $\bu\in\R^{\ell+1}$.
\end{assumption}
This of course implies that the coefficients  $A_{i \to j}^{\ell}$ and $B_i^{\ell}$ 
are almost surely uniformly bounded:
\begin{equation}\label{eq:Abd}
\P\Big(\sup_{(i,j)\in \vE, \ell \ge 0} \big|A_{i \to j}^{\ell}\big| \le K\Big)= 1\, ,
\;\;\;\;\;\;
\P\Big(\sup_{i\in V, \ell \ge 0} \big|B_{i}^{\ell}\big| \le K\Big)= 1\, .
\end{equation}
We say that a function $\psi : \R^n \to \R$ is pseudo-Lipschitz if there is $\cL \ge 0$ such that for every $\bx,\by \in \R^n$, $|\psi(\bx) - \psi(\by)| \le \cL (1+ \|\bx\| + \|\by\|)\|\bx-\by\|$.
\begin{theorem}  \label{thm:clt}
Let $\ell \ge 0$ and $\psi : \R^{\ell+1} \to \R$ be a pseudo-Lipschitz function. Let 
$(U_0,\cdots,U_{\ell})\sim\normal(0,\id_{\ell+1})$. Then for any fixed $\ell$, the following hold
\begin{align}
\plim_{k\to\infty}\frac{1}{k} \sum_{i \in \dj} \psi\big(u_{i \to j}^{0},\cdots,u_{i \to j}^{\ell}\big) &= \E\big\{\psi(U_0,\cdots,U_\ell)\big\}\, , \\
\plim_{k\to\infty}\frac{1}{k} \sum_{i \in \dj} \psi\big(u_{i}^{0},\cdots,u_{i}^{\ell}\big) &= \E\big\{\psi(U_0,\cdots,U_\ell)\big\}\, .
\end{align}
Here, $\plim$ denotes the limit in probability.
\end{theorem}
The next two lemmas will be useful in the proof of the above theorem.  
\begin{lemma}\label{lem:cond1}
For $(i,j) \in \vE(T)$ and $\ell \ge 0$,
\begin{equation}
u^{\ell}_{i\to j} \,\big|\,  \cG_{i \to j}^{\ell-1,-} \, \stackrel{\rmd}{=}\,  \normal\big(0, \big(\tau_{i\to j}^{\ell}\big)^2\big) \, ,
\end{equation}
where $(\tau_{i\to j}^{\ell})$ are defined by the recursion:
\begin{align}\label{eq:update_tau}
\big(\tau_{i\to j}^{\ell+1}\big)^2 = \frac{1}{k-1} \sum_{v \in \di \setminus j} \big(A_{v \to i}^{\ell-1}\big)^2\, \big(\tau_{v\to i}^{\ell}\big)^2 \, ,~~~\mbox{and}~~~
\tau_{i\to j}^{0} = 1 \, .
\end{align}
\end{lemma}
\begin{proof}
We proceed by induction, the base case $\ell=0$ being clear. Suppose the statement holds for all $(i,j) \in \vE$ 
and iterations up to $\ell$. We have
\[u_{i \to j}^{\ell+1} = \frac{1}{\sqrt{k-1}}\sum_{ v \in \di \setminus j} A_{v\to i}^{\ell-1}\, u_{v \to i}^{\ell}.\] 
We observe that for all $v \in \di \setminus j$, $A_{v\to i}^{\ell-1}$ is $\cG_{i \to j}^{\ell,-}$--measurable and $u_{v \to i}^{\ell} \big| \cG_{i \to j}^{\ell,-} \stackrel{\rmd}{=} u_{v \to i}^{\ell} \big| \cG_{v \to i}^{\ell-1,-} \stackrel{\rmd}{=}  \normal(0,\big(\tau_{v\to i}^{\ell}\big)^2)$, where the last inequality in law is by induction. Since $(u_{v \to i}^{\ell})_{v \in \di}$ are independent, we deduce that $u^{\ell+1}_{i\to j}$ is Gaussian conditional on $\cG_{i \to j}^{\ell,+}$, with mean and variance 
\[\E\big\{u^{\ell+1}_{i\to j} \,|\,  \cG_{i \to j}^{\ell,+} \big\} = 0\, ,~~~~ \E\big\{(u^{\ell+1}_{i\to j})^2  \,|\,  \cG_{i \to j}^{\ell,+} \big\} = \frac{1}{k-1} \sum_{v \in \di \setminus j } \big(A_{v \to i}^{\ell-1}\big)^2 \big(\tau_{v\to i}^{\ell}\big)^2\,.\]
This completes the induction argument. 
\end{proof}

\begin{lemma}\label{lem:moment}
For $(i,j) \in \vE(T)$ and $\ell \ge 0$,
\[\E\big\{ \big(u_{i \to j}^{\ell}\big)^2\big\} \le K^{2\ell}.\]
\end{lemma}
\begin{proof}
Since $A_{v \to i}^{\ell-1} u_{v \to i}^{\ell}$ and $A_{v' \to i}^{\ell-1} u_{v' \to i}^{\ell}$ are independent for $v \neq v'$ and have zero mean, we have 
\begin{align*}
\E\big\{ \big(u_{i \to j}^{\ell+1}\big)^2\big\} &= \frac{1}{k-1}\sum_{ v \in \partial i \setminus j} \E\big\{ \big(A_{v \to i}^{\ell-1} u_{v \to i}^{\ell}\big)^2 \big\} \\
&\le K^2 \frac{1}{k-1}\sum_{ v \in \partial i \setminus j}  \E\big\{  (u_{v \to i}^{\ell})^2 \big\} \\
&= K^2\,\E\big\{  (u_{i \to j}^{\ell})^2 \big\} \, ,
\end{align*}
where the second line follows from Assumption~\ref{ass:Bounded}, and the last line follows by transitivity of the tree. Since $\E[  (u_{i \to j}^{0})^2 ] = 1$, we obtain the desired result. 
\end{proof}
\begin{lemma}\label{lem:Delta1}
For $\ell \ge 0$ there exists a constant $C = C(\ell)<\infty$ such that 
\begin{equation}
\E\Big\{\Big(\big(\tau_{i\to j}^{\ell}\big)^2- 1\Big)^2\Big\} \le \frac{C(\ell)}{k-1} \, .
\end{equation}
\end{lemma}
\begin{proof}
We define
\begin{equation*}
\Delta^{\ell} := \E\Big\{\Big(\big(\tau_{i\to j}^{\ell}\big)^2- 1\Big)^2\Big\} \, .
\end{equation*}
Using the relation~\eqref{eq:update_tau}, we have the bound
 \begin{align*}
 \Delta^{\ell+1} &\le 2 \E\Big\{\Big(\frac{1}{k-1} \sum_{v \in \di \setminus j } \big(A_{v \to i}^{\ell-1}\big)^2 \big((\tau^{\ell}_{v\to i})^2  -1\big)\Big)^2\Big\}\\
 &~+2 \E\Big\{\Big(\frac{1}{k-1} \sum_{v \in \di \setminus j } \big(A_{v \to i}^{\ell-1}\big)^2 - 1\Big)^2\Big\} \, .
 \end{align*}
We bound the first term via Jensen's inequality and the fact that $|A_{v \to i}^{\ell-1}| \le K$ almost surely. We expand the square in the second term and use independence together with the normalization condition~\eqref{eq:Normalization} to obtain the recursion 
\begin{align*}
 \Delta^{\ell+1} &\le  \frac{2K^4}{k-1} \sum_{v \in \di \setminus j } \E\Big\{\Big((\tau^{\ell}_{v\to i})^2-1\Big)^2\Big\}
 + \frac{2}{(k-1)^2}  \sum_{v \in \di \setminus j } \E\Big\{\Big(\big(A_{v \to i}^{\ell-1}\big)^2  - 1\Big)^2\Big\}\\
 &\le 2K^4 \Delta^{\ell} + \frac{2 (K^4 +1)}{k-1} \, .
 \end{align*}
Since $\Delta^0 = \big(\E\big\{\big(u^{0}_{i}\big)^2\big\} - 1\big)^2=0$, it follows that $\Delta^\ell \le C(\ell) / (k-1)$. 
\end{proof}

\begin{proof}[Proof of Theorem~\ref{thm:clt}]   
We only prove convergence of $(u^0_{i \to j},\cdots,u^\ell_{i \to j})$. 
Convergence of $(u^0_{i},\cdots,u^\ell_{i})$ can be shown similarly. We proceed by induction over $\ell$. The base case $\ell=0$ follows from the law of large numbers, as $(u^{0}_{i \to j})_{i \in \dj}$ are i.i.d.\ with zero mean and unit variance. We suppose that the claim is true up to $\ell$. Let $\psi : \R^{\ell+2} \to \R$. Since $(u^{s}_{i \to j})_{0 \le s \le \ell}$ are $\cG_{i \to j}^{\ell,-}$--measurable,  Lemma~\ref{lem:cond1} implies 
\begin{equation}\label{eq:condlaw}
\frac{1}{k} \sum_{i \in \dj} \psi\big(u_{i \to j}^{0},\cdots,u_{i \to j}^{\ell+1}\big) \, \Big|\, \cG_{i \to j}^{\ell,-}  ~\stackrel{\rmd}{=}~ \frac{1}{k} \sum_{i \in \dj} \psi\big(u_{i \to j}^{0},\cdots,u_{i \to j}^{\ell}, \tau_{i \to j}^{\ell+1} \, U_i\big) \, ,
\end{equation}
where $(U_i)_{i \in \dj}$ are mutually independent $\normal(0,1)$ r.v.'s which are independent of everything else.
Using Lemma~\ref{lem:Delta1} combined with the facts that $\psi$ is pseudo-Lipschitz and
 $(u^{s}_{i \to j})_{0 \le s \le \ell}$ have moments bounded uniformly in $k$ (because the coefficients $(A_{i\to j}^{s})$
 are almost surely uniformly bounded), we have 
\begin{align}\label{eq:bdabs}
\E\Big\{ \Big|\psi\big(u_{i \to j}^{0},\cdots,u_{i \to j}^{\ell}, \tau_{i \to j}^{\ell+1} \, U_i\big) &- \psi\big(u_{i \to j}^{0},\cdots,u_{i \to j}^{\ell}, U_i\big)\Big|\Big\} 
&\le C(\ell) \Big(\E\big\{\big(\tau_{i \to j}^{\ell+1} -1\big)^2\big\}^{1/2} + \E\big\{\big(\tau_{i \to j}^{\ell+1} -1\big)^2\big\}  \Big)\nonumber\\
&\le \frac{C(\ell)}{\sqrt{k-1}} \, .
\end{align}
Next, we define the function
\[\bar{\psi}(u^0,\cdots,u^\ell) := \E_U\big\{\psi(u^0,\cdots,u^\ell, U)\big\},~~~U \sim \normal(0,1).\]
By independence of the tuples $\big(u_{i \to j}^{0},\cdots,u_{i \to j}^{\ell}, U_i\big)_{i \in \dj}$ we have
\begin{equation}\label{eq:bdvar}
\E\Big\{\Big(\frac{1}{k} \sum_{i \in \dj}  \big\{ \psi\big(u_{i \to j}^{0},\cdots,u_{i \to j}^{\ell},  U_i\big) - \bar{\psi}\big(u_{i \to j}^{0},\cdots,u_{i \to j}^{\ell}\big) \big\}\Big)^2\Big\} = \frac{\sigma_{\psi}^2}{k}\, ,
\end{equation}
where 
\[\sigma_{\psi}^2 := \E\big\{\var\Big(\psi\big(u_{i \to j}^{0},\cdots,u_{i \to j}^{\ell}, U_i\big) \, \big|\, (u^{s}_{i \to j})_{0 \le s \le \ell}\Big)\big\}  \, .\]
It follows from Lemma~\ref{lem:moment} and $\psi$ being pseudo-Lipschitz that $\sigma_{\psi}^2$ is bounded uniformly in $k$. 
Combining~\eqref{eq:condlaw},~\eqref{eq:bdabs} and~\eqref{eq:bdvar} via Chebyshev's inequality we obtain
\[\Big|\frac{1}{k} \sum_{i \in \dj}  \big\{ \psi\big(u_{i \to j}^{0},\cdots,u_{i \to j}^{\ell+1}\big) -  \bar{\psi}\big(u_{i \to j}^{0},\cdots,u_{i \to j}^{\ell}\big) \big\} \Big| \xrightarrow[k \to \infty]{p} 0 \, .\]  
Using the induction hypothesis with the function $\bar{\psi}$ allows us to conclude the argument.    %
\end{proof}

\section{The non-linearities and continuum limit}
\label{sec:ContTime}

In this section we propose a choice of the nonlinearities $F_{\ell}, H_{\ell}$
that is motivated by the algorithm introduced for the SK model in
\cite{montanari2021optimization}. We will show that, by rounding the resulting 
estimate $z_i^L$, we obtain an algorithm that verifies the claim of Theorem \ref{thm:main}.
%

\subsection{General construction of the non-linearities and value on the tree}
\label{sec:non-lin}

We begin by constructing a general class of non-linearities  $F_{\ell}, H_{\ell}$,
which depends on two functions $a,b:[0,1)\to \R$. 
The construction proceeds by defining the sequences $(x^\ell_{i \to j})_{ (i,j)\in \vE(T) , \ell \ge 0}$ 
and $(x^\ell_{i})_{ i\in V(T) , \ell \ge 0}$ recursively as 
\begin{align}
x^{\ell+1}_{i \to j} &= x^\ell_{i \to j} + b(\delta\ell , x^\ell_{i \to j}) \delta + \sqrt{\delta} \, u^{\ell+1}_{i \to j}\, ,~~~~ x^0_{i \to j} = \sqrt{\delta} \, u^0_{i} \, ,\label{eq:x_itoj}\\
x^{\ell+1}_{i} &= x^\ell_{i} + b(\delta\ell , x^\ell_{i}) \delta + \sqrt{\delta} \, u^{\ell+1}_{i}\, ,~~~~ x^0_{i} = \sqrt{\delta}\, u^0_{i} \, ,\label{eq:x_i}
\end{align}
where $(u^\ell_{i \to j})_{ (i,j)\in \vE(T) , \ell \ge 0}$ and $(u^\ell_{i})_{ i\in V(T) , \ell \ge 0}$
are defined by the general message passing iteration of Eqs.~\eqref{eq:Iteration}, \eqref{eq:VertexIteration}
with
\begin{align}
A^{\ell}_{i \to j} = a(\delta \ell , x^{\ell}_{i \to j}) \, ,
~~~\mbox{and}~~~
B^{\ell}_{i} =  a(\delta \ell , x^{\ell}_{i})  \, .
\end{align}
In other words we defined
\begin{align}
F_{\ell}(u^0_{i \to j} ,\cdots, u^\ell_{i \to j}) := a(\delta \ell , x^{\ell}_{i \to j}) \, ,
~~~\mbox{and}~~~
H_{\ell}(u^0_{i} ,\cdots, u^\ell_{i}):= a(\delta \ell , x^{\ell}_{i})  \, ,
\end{align}
where the $x_{i\to j}^{\ell}$, and $x_i^{\ell}$ are given by the recursions 
\eqref{eq:x_itoj}, \eqref{eq:x_i}.

We now analyze the expected value achieved on the tree $T$, 
$\frac{k}{2} \E\big\{z^L_o z^L_v\big\}$
where $(o,v)$ is an edge of the tree, with the above choice of non-linearities. 
\begin{theorem} \label{thm:main2}
Assume the functions $a$, $b$ satisfy the following conditions for some constants 
$\eps>0$, $K, M<\infty$:
\begin{enumerate}
\item $\sup_{x\in\R,t\le 1-\eps}|a(t,x)|\le K$.
\item  $|b(t,x)-b(t,x')|\le M|x-x'|$ for all $t\in [0,1-\eps]$, $x,x'\in\R$.
\end{enumerate}
Let $\eps>0$, $\delta \in [0,1]$ and $L \ge 0$ such that $L \delta \le 1-\eps$. 
Then there exists $C = C(M,K,\delta,\eps) < \infty$  such that for all $k \ge 2$, 
\begin{align}  
\Big| k \E\big\{z_{o}^L z_{v}^L\big\} -  2\sqrt{k-1}  \E \Big\{ \sum_{\ell=2}^L a(\delta \ell , x_o^\ell) \, 
\delta \Big\} \Big| \le  C(M,K,\delta,\eps) \, .
\end{align}
\end{theorem}
Note that condition 1 in the above theorem is a restatement of Assumption~\ref{ass:Bounded}.
We will use the next two lemmas in the proof of the above theorem.  

\begin{lemma}\label{lem:A_B}          
Let $\eps>0$, $\delta \in [0,1]$ and $L \in\N $ be such that $L \delta \le 1-\eps$. 
Under the assumptions of Theorem \ref{thm:main2}, there exists $C = C(M,K,L) < \infty$
such that the following holds. 
For all $k \ge 2$,  and $\ell\le L$
\begin{align*}
\E\big\{\big(x_i^{\ell} - x_{i \to j}^{\ell}\big)^2\big\} \le \frac{C}{k} \, ,~~~\mbox{and}~~~
\E\big\{\big(B_i^{\ell} - A_{i \to j}^{\ell}\big)^2\big\} \le \frac{C}{k} \, .
\end{align*}
\end{lemma}
\begin{proof}
First, we control the difference $u_i^\ell - u_{i \to j}^\ell$. Since
\[u_i^\ell = \sqrt{\frac{k-1}{k}} u_{i \to j}^\ell + \frac{1}{\sqrt{k}} A_{j \to i}^{\ell-2}\, u_{j \to i}^{\ell-1} \, ,\]
we have
\begin{align*}
\E\big\{\big(u_i^{\ell} - u_{i \to j}^{\ell}\big)^2\big\} &= \frac{1}{k} \E\Big\{\Big(A_{j \to i}^{\ell-2}\, u_{j \to i}^{\ell-1} - \frac{1}{\sqrt{k}+\sqrt{k-1}}u_{i \to j}^{\ell}\Big)^2\Big\} \\
&\le \frac{2K^{2\ell}}{k} \, ,
\end{align*}
 where the last line follows from the fact $|A_{j \to i}^{\ell-2}| \le K$ and Lemma~\ref{lem:moment}: $\E\big\{\big(u^\ell_{i \to j}\big)^2\big\}  \le K^{2\ell}$.

We now proceed by induction. For $\ell=0$, we have $u_i^\ell = u_{i \to j}^\ell $. 
Further, for any $\ell\ge  0$, letting $\Delta^{\ell}_{i\to j}:=\E\big\{\big(x_i^{\ell} - x_{i \to j}^{\ell}\big)^2\big\}$,
we get
\begin{align*}
\Delta_{i\to j}^{\ell+1} &
= \E\big\{\big(x^\ell_{i \to j}-x^\ell_{i} + b(\delta\ell , x^\ell_{i \to j}) \delta - b(\delta\ell , x^\ell_{i}) \delta + \sqrt{\delta} \, (u^{\ell+1}_{i \to j}-u^{\ell+1}_{i})\big)^2\big\}\\
&\le 3\Delta_{i\to j}^{\ell} + 
3\E\big\{\big(b(\delta\ell , x^\ell_{i \to j}) - b(\delta\ell , x^\ell_{i})\big)^2\big\} \delta^2 + 
3 \E\big\{\big(u^{\ell+1}_{i \to j}-u^{\ell+1}_{i}\big)^2\big\} \delta \\
&\le 3(1+M^2\delta^2)\Delta_{i\to j}^{\ell}  +\frac{6K^{2(\ell+1)}}{k}\delta\, .
\end{align*}
Since $\Delta_{i \to j}^0 = 0$ and using $L\delta\le 1-\eps \le1$, it holds by the discrete Gr\"onwall inequality that for all $\ell\le L$,
\begin{align*}
\Delta_{i\to j}^{\ell} &\le \frac{C(M,K,L)}{k}
\end{align*}
as claimed.
We then obtain the desired bound for $B_i^{\ell} - A_{i \to j}^{\ell}$ by using the Lipschitz property of $a$. 
\end{proof}

\begin{lemma}\label{lem:approx12}
Under the assumptions of Theorem \ref{thm:main2},
let $\eps>0$, $\delta \in [0,1]$ and $L \in\N $ be such that $L \delta \le 1-\eps$. 
Then there exists $C = C(M,K,\delta,\eps) < \infty$
such that the following holds. 
For all $k \ge 2$,  and $\ell\le L$
\begin{align}
\Big|\E\big\{B^{\ell-1}_{i} B^{\ell}_j A^{\ell-1}_{i\to j} (u^{\ell}_{i\to j})^2 \big\} -  \E\big\{B^{\ell}_j\big\} \Big| &\le \frac{C}{\sqrt{k}} \, ,\label{eq:approx1}\\
\Big|\E\big\{B^{\ell-1}_{i} B^{\ell-1}_j  A^{\ell-2}_{i\to j} A^{\ell-2}_{j\to i} \, u^{\ell-1}_{i\to j} u^{\ell-1}_{j\to i} \big\} \Big| &\le K^6  \, .\label{eq:approx2}
\end{align}
\end{lemma}
\begin{proof}
In the following we omit the dependence of various constants upon $K,M$.
We start with \eqref{eq:approx1}. 
Using Lemma~\ref{lem:Delta1}, and the fact that $B_i^{\ell-1}, A_{i\to j}^{\ell-1} \in 
\cG_i^{\ell-1,+}$ and $B_{j}^{\ell} \in \cG^{\ell,+}_j=\sigma(\cG^{\ell-1,+}_{i\to j}\cup \cG^{\ell,+}_{j\to i})$, 
we have
\begin{align*}
\Big|\E\big\{B^{\ell-1}_{i} B^{\ell}_j A^{\ell-1}_{i\to j} \big((u^{\ell}_{i\to j})^2-1\big) \big\} \Big|
&=\Big|\E\big\{B^{\ell-1}_{i} B^{\ell}_j A^{\ell-1}_{i\to j} \big(\big(\tau_{i\to j}^{\ell}\big)^2 -1\big)  \big\} \Big|\\
&\le \E\Big\{\Big(B^{\ell-1}_{i} B^{\ell}_j A^{\ell-1}_{i\to j}\Big)^2\Big\}^{1/2} \cdot \E\Big\{\Big(\big(\tau_{i\to j}^{\ell}\big)^2 -1\Big)^2 \Big\}^{1/2} \\
&\le \frac{K^3 C(\ell)}{\sqrt{k-1}}\, .   
\end{align*}
Moreover, using Lemma~\ref{lem:A_B} in a similar way, we have
\[\Big|\E\big\{B^{\ell-1}_{i} B^{\ell}_j A^{\ell-1}_{i\to j} \big\} - \E\big\{ A^{\ell}_{j\to i} \big(A^{\ell-1}_{i\to j}\big)^2 \big\} \Big| \le  \frac{C}{\sqrt{k}}\, .\]
Since $A^{\ell}_{j\to i}$ and $A^{\ell-1}_{i\to j}$ are independent, 
\[\E\big\{ A^{\ell}_{j\to i} \big(A^{\ell-1}_{i\to j}\big)^2 \big\} = \E\big\{ A^{\ell}_{j\to i} \big\} \cdot \E\big\{ \big(A^{\ell-1}_{i\to j}\big)^2 \big\} =  \E\big\{ A^{\ell}_{j\to i} \big\}, \]
where we used the normalization condition~\eqref{eq:Normalization} to obtain the last equality. 
We then obtain the bound~\eqref{eq:approx1} with application of Lemma~\ref{lem:A_B}
by noting that, for $\ell\le L$, $C(\ell,\eps)$ is bounded by a constant depending on $\delta$ and $\eps$
(and $K,M$).

In order to prove Eq.~\eqref{eq:approx2}, we use the boundedness assumption~\eqref{eq:Abd} we have
\begin{align*}
\Big|\E\big\{B^{\ell-1}_{i} B^{\ell-1}_j  A^{\ell-2}_{i\to j} A^{\ell-2}_{j\to i} \, u^{\ell-1}_{i\to j} u^{\ell-1}_{j\to i} \big\} \Big| &\le K^4 \E\big\{ \big|u^{\ell-1}_{i\to j} u^{\ell-1}_{j\to i} \big|\big\} \\
&= K^4   \E\big\{ \big|u^{\ell-1}_{i\to j} \big|\big\}\cdot  \E\big\{ \big| u^{\ell-1}_{j\to i} \big|\big\} \\ 
&\le K^6\, , 
\end{align*}
where the second line is by independence, and the third line follows from Lemma~\ref{lem:cond1}
and Lemma~\ref{lem:Delta1}.
\end{proof}

\begin{proof}[Proof of Theorem~\ref{thm:main2}]
Let $\eps>0$, $\delta \in [0,1]$ and $L \le (1-\eps)/\delta$. According to Proposition~\ref{propo:EnergyTree},
\begin{align*}
k\E\big\{z_{o}^L z_{v}^L\big\} &= 2\sqrt{k-1} \, \sum_{\ell=2}^L\E\big\{ A^{\ell-2}_{v\to o}(u^{\ell-1}_{v\to o})^2 B^{\ell-1}_o B^{\ell-2}_{v}\big\} \, \delta 
+ \sum_{\ell=1}^L\E\big\{B^{\ell-1}_{o} B^{\ell-1}_v  A^{\ell-2}_{v\to o} A^{\ell-2}_{o\to v} \, u^{\ell-1}_{v\to o} u^{\ell-1}_{o\to v}\big\}  \, \delta \, .
\end{align*}
Using Lemma~\ref{lem:approx12}, and the formula $B^{\ell}_{o} = a(\delta \ell , x^{\ell}_{o})$ we obtain the desired bound: 
\begin{align*} 
\Big|k\E\big\{z_{o}^L z_{v}^L\big\}  - 2\sqrt{k-1} \, \sum_{\ell=2}^L  \E\big\{a(\delta \ell , x^{\ell}_{o})\big\} \delta\Big| & \le 2 \max_{ 2 \le \ell \le L} C(\ell,\eps) + K^6 \, .
\end{align*}
\end{proof}

\subsection{The Parisi PDE}
\label{subsec:PDE}
At this point we have constructed a discrete martingale $z^{L}_v$ on each node of the graph and we have expressed its edge-correlation in terms of a Riemann sum of the function $a$, see Theorem~\ref{thm:main2} in the previous section. Now it remains to choose the non-linearities in a way which, once the continuum limit $\delta \to 0$ is taken, maximizes this edge-correlation, and at the same time drives this martingale to a binary value $\{-1,+1\}$ for each node $v$. This optimization problem was solved in~\cite{montanari2021optimization,ams20} via the Parisi formula.   
 Here we recall properties of the function $\Phi_{\gamma}$ defined in \eqref{eq:parisiPDE}. The results quoted below are taken from $\cite{ams20}$ which systematically treats the more general setting of non-monotone $\gamma$ at zero or positive temperature. In the positive temperature setting, analogous results have been obtained in for instance \cite{auffinger2015parisi,jagannath2016dynamic,chen2017variational}.

\begin{proposition}\cite[Lemmas 6.2, 6.4]{ams20}
\label{prop:phireg}
For any $\gamma\in\cuU$, the function $\Phi_{\gamma}(t,x)$ is continuous on $[0,1]\times \R$, convex in $x$, and satisfies the following regularity properties for any $\eps>0$.
\begin{itemize}
\item[$(a)$] $\partial_x^j\Phi \in L^{\infty}([0,1-\eps];L^2(\R)\cap L^{\infty}(\R))$ for $j\ge 2$.
\item[$(b)$] $\partial_t\Phi \in L^{\infty}([0,1]\times \R)$ and $\partial_t\partial_x^j\Phi\in L^{\infty}([0,1-\eps];L^2(\R)\cap L^{\infty}(\R))$ for $j\ge 1$.
\item[$(c)$] $|\partial_x\Phi_{\gamma}(t,x)|\leq 1$ for $(t,x)\in [0,1)\times \mathbb R$.
\end{itemize}

\end{proposition}

We now define a corresponding  diffusion process $X_t$ which solves the following SDE with $(B_t)$ a standard Brownian motion

\begin{equation}
\label{eq:parisiSDE1}
\rmd X_t = \gamma(t)\partial_x\Phi_{\gamma}(t,X_t) \rmd t + \rmd B_t\, ,\quad X_0 = 0\, .
\end{equation}

\begin{proposition}{\cite[Lemma 6.5]{ams20}}
\label{prop:Xt}
For any $\gamma\in \cuU$, Eq.~\eqref{eq:parisiSDE1} has a pathwise unique strong solution $(X_t)_{t\in [0,1]}$ which is almost surely continuous and satisfies 
\begin{align}
\partial_x\Phi_{\gamma}(t,X_t) = \int_0^t \partial_{xx}\Phi_{\gamma}(s,X_s)\, \rmd B_s\, .
\label{eq:DxIntegral}
\end{align}
\end{proposition}

Finally we give additional properties for the minimizing $\gamma_*=\arg\min_{\gamma\in\cuU}\Par(\gamma)$. 
\begin{lemma}
\label{lem:identity}
If Assumption~\ref{ass:frsb} holds with minimizer $\gamma_*\in\cuU$, then for all $t\in [0,1)$ the following identities hold:
\begin{align}
    \label{eq:id2} \E\{\partial_{xx}\Phi_{\gamma_*}(t,X_t)^2\} &= 1 \, ,\\
    \label{eq:id4}\int_0^1 \E\big\{\partial_{xx}\Phi_{\gamma_*}(t,X_t)\} \rmd t &=\Par_\star \, .
\end{align}
\end{lemma}

\begin{proof}
The first identity follows by combining \cite[Lemma 6.15]{ams20} and
 \cite[Corollaries 6.10 and 6.11, and Theorem 5]{ams20}. 


To prove the second identity, we use \cite[Lemma 2.9, Equation (2.7)]{sel21} which states that
\[
    \mathbb E\{\partial_{xx}\Phi_{\gamma_*}(t,X_t)\} = \int_t^1 \gamma_*(s)\rmd s \, ,\quad t\in [0,1) \, .
\]
As a result it remains to check that
\[
    \int_0^1 \int_t^1 \gamma_*(s)\rmd s\rmd t=\Par_\star \,.
\]
Changing the order of integration, we find
\begin{align*}
    \int_0^1 \int_t^1 \gamma_*(s)\rmd s\rmd t 
    =
    \int_0^1 \int_0^s \gamma_*(s)\rmd t\rmd s  
    = \int_0^1 s\gamma_*(s)\rmd s \, .
\end{align*}
The desired equality now follows by the last displayed equations in \cite[Proof of Lemma 3.11]{sel21} upon specializing to the case $(\xi(t),h)=(t^2/2,0)$. 
\end{proof}

\subsection{Choosing functions $a(t,x), b(t,x)$ and taking the continuum limit}
\label{sec:continuum}
We next introduce our choice of functions $a(t,x), b(t,x)$.
We begin by introducing a discrete time stochastic process, depending on the 
discretization parameter $\delta$.
 Let $(U^{\delta}_0,U^{\delta}_1,\cdots)$ be a sequence of i.i.d.\ $\normal(0,1)$ random variables, and define the auxiliary sequence 
 \begin{align}
X^{\delta}_{\ell+1} &= X^{\delta}_{\ell} + b(\delta\ell , X^{\delta}_{\ell}) \delta + \sqrt{\delta}\, U^{\delta}_{\ell+1}\, ,~~~~ X^{\delta}_0 = \sqrt{\delta}\, U^{\delta}_0 \, ,\label{eq:Xdelta} \\
 Z^{\delta}_L &= \sqrt{\delta} \, \sum_{\ell=1}^L a\big(\delta (\ell-1), X^{\delta}_{\ell-1}\big) \,  U_\ell^{\delta}\, ,  ~~~~ Z^{\delta}_0 = 0 \, .
  \end{align}

  Let $\Phi_{\gamma}$ solve the Parisi PDE, and $\gamma = \gamma_*: [0,1) \to \R_+$ be the optimal order parameter in the Parisi formula, c.f.\ Section~\ref{sec:intro}. 
  We then define
\begin{align}\label{eq:uv}
b(t,x) &:= \gamma(t) \partial_{x} \Phi_{\gamma}(t,x)\, , \\
\hat a(t,x) &:=  \partial_{xx} \Phi_{\gamma}(t,x) \, ,\\
a(\ell\delta,x) & := \xi_{\delta,k}(\ell\delta) \cdot \hat a(\ell\delta,x)\, ,\;\;\;\;\; 
 \xi_{\delta,k}(\ell\delta) :=\E\big\{\hat a(\ell\delta,x^{\ell}_{\to})^2\big\}^{-1/2}\, .
\end{align}
In the last equation, the expectation is with respect to $x^{\ell}_{\to}$, a random variable distributed 
as  $x^{\ell}_{i\to j}$ for any $(i,j)\in E(T)$. (The distribution of these random 
variables can be defined recursively using Eq.~\eqref{eq:RDE} together with the relation between $x^{\ell}_{i\to j}$
and $u^{\ell}_{i\to j}$ in Eq.~\eqref{eq:x_itoj}.) 
Further notice that we defined $a(t,x)$ only for $t\in \N\delta$, which is what is required by
the algorithm construction. We can extend this definition 
by letting $a(t,x) := \xi_{\delta,k}(\lceil t/\delta\rceil \delta) \cdot \hat a(t,x)$.
for all values of $t$.

By an application of Theorem \ref{thm:clt}, we have the following.
\begin{proposition}\label{prop:KinftyEnergy}
Let $(X^{\delta}_{\ell})_{\ell\ge 0}$ be defined as per Eq.~\eqref{eq:Xdelta}.
 Then we have
\begin{align}
\lim_{k\to\infty}  \xi_{\delta,k}(\ell\delta)  &= \xi_{\delta,*}(\ell\delta):= \E\{\hat a(\ell\delta,X^{\delta}_{\ell})^2\}^{-1/2}\, ,\\
\lim_{k\to\infty}\frac{k}{2\sqrt{k-1}} \E\big\{z_{o}^L z_{v}^L\big\} & =   \sum_{\ell=2}^L \xi_{\delta,*}(\ell\delta)
\, \E \{\hat a(\delta \ell , X^{\delta}_\ell)\}\delta\, .\label{eq:KinftyEnergy}
\end{align}
\end{proposition}
  
  The  limit $\delta \to 0$ of the processes 
  introduced in Eq.~\eqref{eq:Xdelta}  is given by
   $(X_t, Z_t)_{t \in [0,1]}$, where $(X_t)_{t\in [0,1]}$ solves the SDE in Eq.~\eqref{eq:parisiSDE1} which we reproduce here:
 \begin{equation}\label{eq:parisiSDE2}
 \rmd X_t = \gamma(t) \partial_x \Phi_{\gamma}(t,X_t ) \rmd t + \rmd B_t \, ,~~~~ X_0 = 0\, ,
 \end{equation}
and $(Z_t)_{t \in [0,1]}$  is defined by
 \begin{equation}\label{eq:Z-cont}
 Z_t = \int_0^t \partial_{xx} \Phi_{\gamma}(s,X_s ) \rmd B_s \, .
 \end{equation}
 In the above, $(B_t)_{t \in [0,1]}$ is a standard Brownian motion.

 The following result now shows convergence of the limiting discrete-time dynamics to a continuous-time 
 stochastic differential equation as $\delta\to 0$. It follows from the proof of \cite[Proposition 5.3]{ams20} with \eqref{thm:clt} replacing state evolution for AMP. The idea there is to show by induction on $\ell$ and \eqref{eq:normalization-limit} that the distinction between $a$ and $\hat a$ becomes irrelevant as $k\to\infty$ in the sense that the normalizing constants in \eqref{eq:normalizeF} each tend to $1$ (see \cite[Eq.\ (5.7)]{ams20}). One shows in the same induction that the discretized dynamics approximate the SDE \eqref{eq:parisiSDE2}.
\begin{proposition}\label{prop:SDEdisc}
For any $\eta>0$ there exist constants $\delta_0 \in [0,1]$ and $C<\infty$ and a coupling between $\{(U^{\delta}_{\ell},X^{\delta}_{\ell})_{\ell\geq 0}$ and $\{(B_t,X_t)\})_{t\geq 0}$ such that for all $\delta\leq \delta_0$,
\begin{equation}\label{eq:X-converge}
    \sup_{t\in [0,1-\eta]} \E\big\{|X^{\delta}_{\lfloor t/\delta\rfloor}-X_{t}|^2\big\} \leq C\delta \, , 
    ~~~ \mbox{and}~~~
    \sup_{t\in [0,1-\eta]} \E\big\{|Z^{\delta}_{\lfloor t/\delta\rfloor}-Z_{t}|^2\big\} \leq C\delta \, .
\end{equation}
Moreover for all $L \le (1-\eta)/\delta$,
\begin{equation}\label{eq:ahat-converge}
    \limsup_{k\to\infty}\sup_{\ell\in [L],x\in\mathbb R} \E\big\{|a(\ell\delta,x)-\hat a(\ell\delta,x)| \big\} \leq C\sqrt{\delta}\, .
\end{equation}
\end{proposition}

We next compute the value achieved in the limit $\delta \to 0$.
\begin{proposition}\label{prop:discr}
Let $\eta \in [0,1)$, $\delta \in (0,1)$ and $L = \lfloor (1-\eta)/\delta \rfloor$.
\[ \lim_{\delta \to 0}\,  \lim_{k \to \infty} \, \frac{k}{2\sqrt{k-1}} \E\big\{z_{o}^L z_{v}^L\big\} =  \int_0^{1-\eta} \E \big\{\partial_{xx} \Phi_{\gamma}(s,X_s) \big\} \rmd s \ge \Par_\star - \eta \, ,\]
where $(X_t)_{t \in [0,1]}$ solves the SDE~\eqref{eq:parisiSDE2}.
\end{proposition}
\begin{proof}
By Proposition \ref{prop:KinftyEnergy}, we need to evaluate  the $\delta\to 0$ limit of the
right-hand side of Eq.~\eqref{eq:KinftyEnergy}.

Recall that under Assumption~\ref{ass:frsb}, \eqref{eq:id2} ensures that, for all $t\in[0,1)$:
 \begin{equation}
    \label{eq:normalization-limit}
    \E\{\hat a(\delta \ell,X_t)^2\}=1.
 \end{equation}
Since $\hat a(t,x)$ is Lipschitz in $x$ uniformly over $t$ bounded away from $1$, Proposition \ref{prop:SDEdisc} implies
\begin{align*}
\lim_{\delta\to 0}\max_{\ell\le L} \big|\xi_{\delta,*}(\ell\delta)-1\big| = 0\, .
\end{align*}
Therefore, for any $t\in [0,1)$,
\begin{align*}
 \lim_{\delta\to 0}\sum_{\ell=2}^{\lfloor t/\delta\rfloor} \xi_{\delta,*}(\ell\delta)
\, \E \{ \hat a(\delta \ell , X^{\delta}_\ell)\}\delta &=  \lim_{\delta\to 0}\sum_{\ell=2}^{\lfloor t/\delta\rfloor} 
\, \E \{ \hat a(\delta \ell , X^{\delta}_\ell)\}\delta\\
&=  \lim_{\delta\to 0}\sum_{\ell=2}^{\lfloor t/\delta\rfloor} 
\, \E \{ \hat a(\delta \ell , X_{\ell\delta})\}\delta \\ 
    &= \int_0^t \E \big\{\partial_{xx} \Phi_{\gamma}(s,X_s) \big\} \rmd s \, . 
 \end{align*}

Therefore by appealing to Proposition \ref{prop:KinftyEnergy},
 \[
    \lim_{\delta \to 0} \, \lim_{k\to \infty} \, \frac{ k \E\big\{z_o^L z_v^L\big\}}{2\sqrt{k-1}}  = \int_0^{1-\eta} \E \big\{\partial_{xx} \Phi_{\gamma}(s,X_s) \big\} \rmd s \stackrel{\eqref{eq:id4}}{\geq} \sfP_\star - \eta  \, ,
 \]
where in the final step we used that for $t\in [1-\eta,1]$,
\begin{align*}
    \E \big\{\partial_{xx} \Phi_{\gamma}(s,X_s) \big\}&\leq \sqrt{\E \big\{\partial_{xx} \Phi_{\gamma}(s,X_s)^2 \big\}} \stackrel{\eqref{eq:id2}}{\leq} 1.
\end{align*}
\end{proof}


 \subsection{Rounding and proof of Theorem~\ref{thm:main}}
 Let $G = (V,E)$ be a $(L,\eps)$-locally treelike $k$-regular graph, as 
 per Definition~\ref{def:locallytreelike}. The message passing algorithm described in the previous section
 outputs a vector $\bz^L = (z_i^L)_{i \in V}$ after $L$ iterations. In this section we relate the 
 value obtained on the finite graph $G$ with that achieved on the infinite tree $T$, describe a 
 rounding procedure to obtain feasible solutions $\bsigma^1,\, \bsigma^2  \in \{-1,+1\}^n$ 
 and complete the proof of Theorem~\ref{thm:main}. 

Let $\hat{\bz}^L \in [-1,+1]^n$ defined by
\begin{align}
\hat{z}_i^L = 
\begin{cases}
z_i^L & \mbox{if } |z_i^L| \le 1 \, ,\\
\sign(z_i^L) & \mbox{otherwise .} 
\end{cases}
\end{align}

The algorithm generates $\hat{\bsigma} \in \{-1,+1\}^n$ by 
drawing $(\hat{\sigma}_i)_{i\in V}$ conditionally independently given $\bu^0= (u^{0}_i)_{i\in V}$
so that 
\begin{equation}
\E\big\{\hat{\sigma}_i \big| \bu^0\big\} = \hat{z}_i^L\, .
\end{equation}
Equivalently, we can let $\hat{\sigma}_i =\sign(\hat{z}_i^L-\tilde{u}^0_i)$
where $(\tilde{u}^0_i)_{i\in V}$ is a collection of i.i.d.\ random variables uniformly distributed in $[-1,1]$ and independent of 
$(u^{0}_i)_{i\in V}$.
Observe that this algorithm is $L$-local and balanced as per Definition~\ref{def:localalg}, and recalling that $U_{G} (\bsigma) = (1/n) \sum_{(i,j)\in E(G)} \sigma_i \sigma_j$, we have
\begin{equation}\label{eq:eqU}
    \E U_G( \hat{\bsigma}) = \E U_G(\hat{\bz}^L) \, .
\end{equation}
Furthermore, we have the following concentration properties.
\begin{lemma}\label{lem:concent}
There exists a constant $C = C(L,k) >0$ such that for all $t \ge 0$,
\begin{align}
\P\Big( \big|\frac{1}{n}\sum_{i \in V(G)} \hat{\sigma}_i \big| \ge t +\eps\Big) &\le C e^{- n t^2/(2C)} \, , \label{eq:sumconc}\\
\P\Big( \big|U_G(\hat{\bsigma}) -  \E U_G(\hat{\bsigma})\big| \ge k(t + \eps)\Big) &\le C e^{- n t^2/(2C)} \, .\label{eq:UGconc}
\end{align}
\end{lemma}
\begin{proof}
We first prove Eq.~\eqref{eq:sumconc}. Let $W(G)\subseteq V(G)$ consist of all vertices with $L$-neighborhood a tree.
By a simple greedy argument, $W(G)$ can be partitioned into $C_{L,k}$ sets $W_1,\dots,W_{C_{L,k}}\subseteq W(G)$ of vertices such that $d_{G}(i,j) \ge 2L +1$ for any $i,j\in W_q$ in the same set. 
Then for any $1 \le q \le C_{L,k}$, $L$-locality implies that
\[
    S_q=\sum_{i\in W_q}\hat{\sigma}_i
\]
is a sum of i.i.d.\ zero-mean $\pm 1$ random variables. By Hoeffding's inequality, $\P(|S_q| \ge t) \le 2e^{- t^2/(2|W_q|)}$ for all $t \ge 0$. Since
\[
    S := \sum_{i\in W(G)}\hat{\sigma}_i=\sum_{q = 1}^{C_{L,k}}  S_q,
\] 
we obtain by a union bound
\[
    \P\big(|S| \ge t\big) \le \sum_{q =1}^{C_{L,k}} \P(|S_q| \ge t/C_{L,k}) \le 2C_{L,k}e^{-t^2/(2n C_{L,k}^2)}\, ,
\]
where we have upper-bounded $|W_q|$ by $n$ to obtain the last inequality. Finally, since the graph $G$ is $(L,\eps)$-locally treelike, 
\[
    \P\big(\big|\sum_{i \in V(G)} \hat{\sigma}_i \big|\ge t\big) \le \P\big(|S| \ge t -\eps n\big) \le 2C_{L,k}e^{-(t-\eps n)^2/(2n C_{L,k}^2)}\, .
\]
Next, we prove Eq.~\eqref{eq:UGconc} using the same approach. We consider a partitioning of $W(G)$ into $C_{L,k}$ sets $W_1,\dots,W_{C_{L,k}}\subseteq W(G)$ of vertices such that $d_{G}(i,j) \ge 2L +3$ for any $i,j\in W_q$ for any $q$. Let 
\[ 
    S_q = \sum_{i \in W_q} \hat{\sigma}_i \sum_{j \in \partial i} \hat{\sigma}_j \, , ~~~\mbox{and}~~~ S = \sum_{q=1}^{C_{L,k}} S_q \, .
\]
 Since $d_{G}(u,v) \ge 2L +3$ implies $d_{G}(i,j) \ge 2L +1$ for any two vertices $i \in \partial u$ and $j \in \partial v$, it follows that $S_q$ is a sum of i.i.d.\ random variables which are also bounded in absolute value by $k$. Hoeffding's inequality implies $\P(|S_q - \E\{S_q\}| \ge t) \le 2e^{- t^2/(8k^2|W_q|)}$ for all $t \ge 0$. We finish the argument similarly, using that $| 2 n U_G(\hat{\bsigma}) - S| \le k \eps $.      
\end{proof}
The next lemma is useful to control error terms in the rounding 
process. 
\begin{lemma}\label{lem:TreeEigenvalue}
Let $\{(\xi_v,\txi_v)\}_{v\in V(T)}$ be a block factor of IID on the infinite $k$-regular 
tree $T$ taking values in $\R^2$, such that $\E\{\xi_v\}=\E\{\txi_v\}=0$ (see Definition~\ref{def:blockfactor}).
For $(u,v)\in E(T)$ an edge, we have 
\begin{align}
\big|\E\{\xi_u\xi_v\}\big|\le \frac{2\sqrt{k-1}}{k}\, \E\{\xi_u^2\}\, ,\;\;\;\;\;
\big| \E\{\xi_u\txi_v\} \big| \le \frac{2\sqrt{k-1}}{k}\, \E\{\xi_u^2\}^{1/2}\E\{\txi_v^2\}^{1/2}\, .
\label{eq:GeneralXi}
\end{align}
\end{lemma}
\begin{proof}
A more general statement appears in~\cite{backhausz2015ramanujan}, but we give a simple 
argument here for completeness. 
Let $(F,\tF): \cG_*(k,\ell;\Omega)\to \R^2$ be the function that witnesses $\{(\xi_v,\txi_v)\}_{v\in V(T)}$ being a block factor of IID.
For $\cS\subseteq \N$ a diverging sequence, 
let $(G_{n})_{n\in \cS}$ be a sequence of $k$-regular graphs of girth larger than $2\ell+2$, such that $\lim_{n\to\infty}\lambda_2(G_n)\le 2\sqrt{k-1}$, $\lim_{n\to\infty}\lambda_n(G_n)\le -2\sqrt{k-1}$.
 Here $k=\lambda_1(G)\geq \lambda_2(G)\geq\dots\geq\lambda_n(G)$ are the (real) eigenvalues of the adjacency matrix of $G$, and the condition on $(\lambda_2,\lambda_n)$ amounts to the $G_n$ being nearly Ramanujan.
 Such a sequence can be constructed by sampling $k$-regular graphs conditional
 on the girth being at least $2\ell+2$ (an event that has probability bounded away from zero
 for each constant $\ell$), and applying Friedman's theorem \cite{friedman2008proof}.
 
 We then construct random variables  $(\omega_i)_{i\in V(G_n)}\sim_{iid}\nu$,
 and  $\{(z_v,\tz_v)\}_{v\in V(G_n)}$, by
 letting $z_v = F(\Ball_v(\ell),\bomega_{\Ball_v(\ell)})$,  
 $\tz_v = \tF(\Ball_v(\ell),\bomega_{\Ball_v(\ell)})$, where $v\in V(G)$ and $\Ball_v(\ell)$
 is the ball of radius $\ell$ in $G_n$. 
 Note that, for any edge $(i,j)\in E(G_n)$, we have $\E\{\xi_u\xi_v\} = \E\{z_iz_j\}$,
 $\E\{\xi_u\}= \E\{z_i\} = 0$, and $\E\{\xi^2_u\}= \E\{z^2_i\}$.
 Let $\bQ\in\R^{n\times n}$ be the second moment matrix $Q_{ij} = \E\{z_iz_j\}$.
 Note that $\<\bfone,\bQ\bfone\>\le C(\ell)n$.
 Letting $\bA_n$ be the adjacency matrix of $G_n$,
 we have
 \begin{align*}
 \E\{\xi_u\xi_v\} & = \frac{1}{2|E(G_n)|}\E\{\<\bz,\bA_n\bz\>\} \\
 & = \frac{1}{nk} \Tr\big(\bA_n \bQ\big) \\
 &\le \frac{1}{nk}\lambda_2(G_n)\Tr(\bQ)+\frac{1}{n^2 k} \lambda_1(G_n) \<\bfone,\bQ\bfone\>\\
 &\le \frac{2\sqrt{k-1}}{k}\E\{\xi_v^2\}+ \frac{C'(\ell)}{n}\, .
\end{align*}
By taking the limit $n\to\infty$ limit, we obtain $\E\{\xi_u\xi_v\}\le 2k^{-1}\sqrt{k-1}\E\{\xi_v^2\}$.
By a similar argument (using bipartite graphs) we deduce 
$\E\{\xi_u\xi_v\}\ge -2k^{-1}\sqrt{k-1}\E\{\xi_v^2\}$. The second inequality
in \eqref{eq:GeneralXi} follows by applying the first inequality to the processes $z_u^{+} = \xi_u + \lambda \txi_u$ 
and $z_u^{-}= \xi_u - \lambda \txi_u$, subtracting the two inequalities, and setting 
$\lambda = \sqrt{\E\{\xi_u^2\}/\E\{\txi_u^2\}}$ to obtain the upper bound, and 
$\lambda= -  \sqrt{\E\{\xi_u^2\}/\E\{\txi_u^2\}}$ to obtain the lower bound.
\end{proof}
Next we build on the results of the previous sections to relate  $\E U_G( \hat{\bsigma})$ to the Parisi formula.  
\begin{lemma}\label{lem:diff}
Let $(\root(T), v(T))$ be a fixed edge in the infinite $k$-regular tree $T$,
and denote by $z^L_{\root(T)}$, $z^L_{v(T)}$, the estimates of the message passing
algorithm on $T$. 
For $\eta > 0$ there exists $\delta_0(\eta), C(\eta) > 0, n_0 \ge 1$ such that for all 
$\delta \le \delta_0(\eta)$, there exists $k_0 = k_0(\eta,\delta)$ such that the following holds for all 
$k \ge k_0$. Let $G$ be $(L,\eps)$-locally treelike $k$-regular graph on $n \ge n_0$ vertices with $L = \lfloor(1-\eta)/\delta\rfloor$. We have 
\begin{align}
\Big| \E U_G(\hat{\bz}^L) - \frac{k}{2}\E\{z^L_{\root(T)} z^L_{v(T)}\} \Big| \le C(\eta)\sqrt{(k-1)\delta} + 2 k\eps \, .
\end{align} 
\end{lemma}
\begin{proof}
Let $E_0\subseteq E$ be the subset of edges $(i,j)\in E$ such that 
$\Ball_i(\ell)$ and $\Ball_j(\ell)$ are both $k$-regular trees. Define
\begin{align*}
U_{G,0}(\bsigma):= \frac{1}{n}\sum_{(i,j)\in E_0}\sigma_i\sigma_j\, .
\end{align*}
Note that $|E\setminus E_0|\le kn\eps$ and therefore 
$\big|\E U_{G}(\hbz^L)-\E U_{G,0}(\hbz^L)\big|\le k\eps$. On the other hand, for any 
$(i,j)\in E_{0}$ the pair of random variables $(z_i^L,z_j^L)$ is distributed as
$(z_{\root(T)}^L,z_{v(T)}^L)$, and therefore
\begin{align*}
\Big| \E U_G(\hat{\bz}^L) - \frac{k}{2}\E\{z^L_{\root(T)} z^L_{v(T)}\} \Big|& \le
\Big| \E U_{G,0}(\hat{\bz}^L) - \frac{k}{2}\E\{z^L_{\root(T)} z^L_{v(T)}\} \Big| +k\eps\\
&= \Big| \frac{|E_0|}{n}\E\{\hz^L_{\root(T)} \hz^L_{v(T)}\} - \frac{k}{2}\E\{z^L_{\root(T)} z^L_{v(T)}\} \Big| +k\eps\\
&\le \frac{k}{2}
\big| \E\{(z^L_{\root(T)}-\hz^L_{\root(T)}) (z^L_{v(T)}-\hz^L_{v(T)})\} 
+2\E\{(z^L_{\root(T)} - \hz^L_{\root(T)})\hz^L_{v(T)}\} \big| +
2k\eps\\
&\le \sqrt{k-1}\E\big\{(z^L_{\root(T)} -\hz^L_{\root(T)})^2\big\} +
2\sqrt{k-1} \E\big\{(z^L_{\root(T)} -\hz^L_{\root(T)})^2\big\}^{1/2}+2k\eps\,.
\end{align*}
Here the last step follows from Lemma~\ref{lem:TreeEigenvalue}. 
Next, we define the function $\psi(x) = \min_{y \in [-1,1]} (x-y)^2$. Since 
$\E\big\{\big(\hat{z}_o^L - z_o^L\big)^2\big\} = \E\psi(z_o^L)$,
we obtain by Theorem~\ref{thm:clt},
\begin{align*}
\lim_{k \to \infty} \E\big\{\big(\hat{z}_o^L - z_o^L\big)^2\big\} &= \E\psi\big(Z^{\delta}_L\big)\\
&\stackrel{(a)}{=}\big|\E\psi\big(Z^{\delta}_L\big) - \E\psi\big(Z_{\delta L}\big) \big| \\
&\stackrel{(b)}{\le} C(\eta)\delta\, .
\end{align*}
Here in $(a)$ we are using the fact that $|Z_t| \le 1$ a.s. for all $t$, which follows by the definition \eqref{eq:Z-cont} of $Z_t$ combined with \eqref{eq:DxIntegral} and Proposition~\ref{prop:phireg}, part $(c)$. Indeed these imply that $\psi(Z_{\delta L})=0$. In step $(b)$ we used
 Proposition~\ref{prop:SDEdisc}.
The desired bound follows by combining the last two displays. 
\end{proof}

\begin{lemma}\label{lem:UG}
For $\eta > 0$ there exists $\delta_0(\eta) > 0$ and $n_0(\eta) \ge 1$ such that for all $\delta \le \delta_0(\eta)$, 
 there exists $k_0 = k_0(\eta,\delta)$ such that the following holds for all $k \ge k_0$.
 Let $G$ be a $(L,\eps)$-locally treelike $k$-regular graph on $n \ge n_0$ vertices with $L = \lfloor(1-\eta)/\delta\rfloor$. We have
\[ \frac{\E U_G(\hbz^L)}{\sqrt{k-1}}  \geq \sfP_\star - 2\eta  - 4 \sqrt{k}\, \eps\, , \] 
\end{lemma}
\begin{proof}
For any $\delta<\delta_0(\eta)$, $k \ge k_0(\eta,\delta)$ as per Lemma \ref{lem:diff},
we have 
\begin{align*}
\frac{\E U_G(\hbz^L)}{\sqrt{k-1}} & \ge  \frac{1}{\sqrt{k-1}} \Big(\frac{k}{2}\E\{z^L_{\root(T)} z^L_{v(T)}\} 
-C(\eta)\sqrt{(k-1)\delta}-2k\eps \Big)\\
&\ge \Par_{\star}-\eta-{\rm err}(k,\delta)-C(\eta)\sqrt{\delta}-4\sqrt{k}\eps\, , 
\end{align*}
where $\lim_{\delta\to 0}\lim_{k\to\infty}{\rm err}(k,\delta) = 0$ by Lemma
\ref{prop:discr}. We can decrease $\delta_0(\eta)$ so that $C(\eta)\sqrt{\delta}\le \eta/3$
and \[\limsup_{k\to\infty} \, {\rm err}(k,\delta)\le \eta/3\]
for all $\delta\le \delta_0(\eta)$.
Finally, we increase $k_0(\eta,\delta)$ such that 
\[{\rm err}(k,\delta)-\limsup_{k\to\infty}\, {\rm err}(k,\delta)\le \eta/3 \, .\]
\end{proof}

We are now ready to prove Theorem~\ref{thm:main}.
\begin{proof}[Proof of Theorem~\ref{thm:main}]
 We obtain a feasible solution $\bsigma^1\in \{-1,+1\}^n$ to the minimum 
bisection problem by using the message passing algorithm with
non-linearities $\{F_\ell , H_\ell\}_{\ell \le L}$ as described in 
Section~\ref{sec:continuum} and the rounding procedure described above. 
We obtain a vector $\tilde{\bsigma}^1 \in \{-1,+1\}^n$ which by Lemma~\ref{lem:concent}, Eq.~\eqref{eq:eqU}, and Lemma~\ref{lem:UG} is such that  
\begin{align*}
\frac{U_G(\tilde{\bsigma}^1)}{\sqrt{k-1}} \ge  \sfP_\star - 2\eta - 8 \sqrt{k} \eps\, , ~~~
\mbox{and}~~~ \big|\sum_{i\in V(G)} \tilde{\sigma}^1_i\big| \le 2\eps  n \, ,
\end{align*}
 with probability at least $1-Ce^{-Cn\eps^2}$, with $C=C(L,k)$. To obtain an exactly balanced partition $\bsigma^1$ (i.e., $\sum_i \sigma_i^1 = 0$),
 it suffices to change the signs of $\eps n$  many vertices. The effect of this 
 change on the normalized cut value $U_G$ is no more than $k \eps$. We then choose $\eps= \eta/\sqrt{k}$ to obtain $U_G(\bsigma^1)/\sqrt{k-1} \ge \sfP_\star - 10 \eta$.    

Next, a feasible solution $\bsigma^2 \in \{-1,+1\}^n$ to the maximum cut problem with $-U_G(\bsigma^2)/\sqrt{k-1} \ge \sfP_\star - 10\eta$
is obtained similarly, by rounding the output of the message passing algorithm 
with non-linearities $\{-F_\ell , H_\ell\}_{\ell \le L}$. (Note that the resulting partition need not be balanced in this case.)
\end{proof}

\subsection*{Acknowledgements}
A.M. was partially supported
by the NSF grant CCF-2006489 and the ONR grant N00014-18-1-2729. M.S. was partially supported by an NSF graduate research fellowship and a Stanford graduate fellowship.
Parts of this work were completed while the authors were virtually participating in the Fall 2020 program on Probability, Geometry and Computation in High Dimensions at the Simons Institute for the Theory of Computing.    

\bibliographystyle{alpha}
\providecommand{\bysame}{\leavevmode\hbox to3em{\hrulefill}\thinspace}
\providecommand{\MR}{\relax\ifhmode\unskip\space\fi MR }
\providecommand{\MRhref}[2]{%
  \href{http://www.ams.org/mathscinet-getitem?mr=#1}{#2}
}

\newcommand{\etalchar}[1]{$^{#1}$}

\end{document}